\documentclass[12pt]{amsart}

\usepackage{amssymb, amscd}
\usepackage{latexsym,epsfig}
\usepackage[all]{xy}
\usepackage{pst-all}
\usepackage{color}

\usepackage[
        hypertexnames=false,
        hyperindex,
        pagebackref,
        pdftex,
        breaklinks=true,
        bookmarks=false,
        colorlinks,
        linkcolor=blue,
        citecolor=red,
        urlcolor=red,
]{hyperref}

\numberwithin{equation}{section}
\def\today{\ifcase\month\or Jan\or Febr\or  Mar\or  Apr\or May\or Jun\or  Jul\or Aug\or  Sep\or  Oct\or Nov\or  Dec\or\fi \space\number\day, \number\year}

\newcommand{\CC}{\mathbb C}
\newcommand{\EE}{\mathbb E}

\newcommand{\GG}{\mathbb G}

\newcommand{\MM}{\mathbb M}

\newcommand{\PP}{\mathbb P}
\newcommand{\QQ}{\mathbb Q}
\newcommand{\RR}{\mathbb R}

\newcommand{\ZZ}{\mathbb Z}

\newcommand\langepijl[1]{\buildrel {#1} \over \longrightarrow}
\newcommand{\Sym}{{\mathrm{Sym}}}

\makeatletter
\def\legendre@dash#1#2{\hb@xt@#1{%
  \kern-#2\p@
  \cleaders\hbox{\kern.5\p@
    \vrule\@height.2\p@\@depth.2\p@\@width\p@
    \kern.5\p@}\hfil
  \kern-#2\p@
  }}
\def\@legendre#1#2#3#4#5{\mathopen{}\left(
  \sbox\z@{$\genfrac{}{}{0pt}{#1}{#3#4}{#3#5}$}%
  \dimen@=\wd\z@
  \kern-\p@\vcenter{\box0}\kern-\dimen@\vcenter{\legendre@dash\dimen@{#2}}\kern-\p@
  \right)\mathclose{}}

\numberwithin{equation}{section}
\newtheorem{theorem}{Theorem}[section]
\newtheorem{lemma}[theorem]{Lemma}
\newtheorem{proposition}[theorem]{Proposition}
\newtheorem{corollary}[theorem]{Corollary}

\newtheorem{definition-lemma}[theorem]{Definition-Lemma}
\theoremstyle{definition}

\newtheorem{vergelijking}[theorem]{Equation}

\theoremstyle{remark}
\newtheorem{remark}[theorem]{Remark}

\begin{document}

\title[Picard modular forms by means of invariant theory]{Generating Picard modular forms \\
by means of invariant theory}
\author{Fabien Cl\'ery}
\address{Institute of Computational and Experimental 
Research in Mathematics,
121 South Main Street, Providence, RI 02903, USA}
\email{cleryfabien@gmail.com}

\author{Gerard van der Geer}
\address{Korteweg-de Vries Instituut, Universiteit van
Amsterdam, Postbus 94248,
The Netherlands}
\email{g.b.m.vandergeer@uva.nl}


\maketitle
\hfill{\sl Dedicated to Don Zagier on the occasion of his 70th birthday \quad}

\begin{abstract}
We use the description of the Picard modular surface for discriminant $-3$
as a moduli space of curves of genus $3$ to generate all vector-valued 
Picard modular forms from bi-covariants for the action of ${\rm GL}_2$
on the space of pairs of
binary forms of bidegree $(4,1)$. The universal binary 
forms of degree $4$
and $1$ correspond to a meromorphic modular form of weight $(4,-2)$ 
and a holomorphic Eisenstein series of weight $(1,1)$.
\end{abstract}

\begin{section}{Introduction}\label{sec-intro}
Some Shimura varieties can be interpreted as moduli spaces of curves
and such an interpretation offers extra ways to study these Shimura
varieties. More precisely, in a number of cases a dense open part
of the Shimura variety is the image of a moduli space of curves under
a morphism of finite degree. Examples are the moduli space of 
principally polarized abelian varieties of dimension $2$ (resp.\  $3$)
where we have the Torelli map $\mathcal{M}_2 \to \mathcal{A}_2$
(resp.\ $\mathcal{M}_3 \to \mathcal{A}_3$)
from the moduli space
of curves of genus $2$ (resp.\ ~$3$).
Igusa \cite{Igusa1967}  used
this to describe the generators for the rings of scalar-valued 
Siegel modular forms of degree $2$ and  later Tsuyumine \cite{Tsuyumine} 
extended this to the case of degree $3$. 
In joint work with Carel Faber \cite{CFG1,CFG2} 
we used the description of 
$\mathcal{M}_2$ as a stack quotient of ${\rm GL}_2$
to extend the work of Igusa 
by describing how invariant theory makes it possible to efficiently 
generate all vector-valued Siegel modular forms (of level~$1$) 
of degree $2$ from one universal vector-valued meromorphic Siegel modular form
$\chi_{6,-2}$ and one scalar-valued holomorphic form $\chi_{10}$. 
Similarly in \cite{CFG3} 
we used the description of an open part of $\mathcal{M}_3$
as a stack quotient of ${\rm GL}_3$ to generate all Siegel and
Teichm\"uller modular forms from a universal 
meromorphic Teichm\"uller modular form $\chi_{4,0,-1}$ of genus $3$
and the form $\chi_9$, a square root of a 
Siegel modular form $\chi_{18}$.
These universal vector-valued modular forms $\chi_{6,-2}$ for genus $2$
and $\chi_{4,0,-1}$ for genus $3$
can be seen as giving the equation
of the universal curve over the moduli space while the scalar-valued
ones $\chi_{10}$ and $\chi_{18}$ are related to the discriminants
of these equations.

It is natural to try to extend this to other Shimura varieties.
In \cite{Shimura1964} Shimura gave a list of arithmetic ball quotients 
that are moduli spaces of curves. This list was extended to a 
complete list by Rohde, see \cite{Rohde,M-O}. 

Here we treat one case of Shimura's list,
a quotient of the $2$-ball that gives the moduli of genus $3$
curves that are triple cyclic covers of the projective line. The
period domain of such curves is a Picard modular surface 
associated to the group of unitary similitudes ${\rm GU}(2,1,{\QQ}(\sqrt{-3}))$. These periods were first studied by Picard in  the late
19th century in a series of papers \cite{P1,P2,P3}.

We show how all vector-valued modular forms on the moduli space 
in question can be generated by invariant theory from 
two universal modular forms, one meromorphic form $\chi_{4,-2}$ of 
weight $(4,-2)$, and a holomorphic Eisenstein series $E_{1,1}$
of weight $(1,1)$. Multiplication of $\chi_{4,-2}$ 
by the scalar-valued modular form $\zeta$,
related to the discriminant, makes $\chi_{4,-2}$ holomorphic. 
These three forms
are Teichm\"uller modular forms, 
but can be viewed as Picard modular
forms on an appropriate congruence subgroup. 
The two vector-valued forms $\chi_{4,-2}$ and $E_{1,1}$ 
can be interpreted as the
quartic and the linear term $f_4$ and $f_1$ in the equation 
of the universal canonical curve over the moduli space
$$
y^3f_1=f_4 \, .
$$

Like in the cases of Siegel modular forms of degree $2$ and $3$, 
the  interpretation of our moduli space as a stack quotient
enables the use of invariant theory. 
This moduli space is a stack quotient of a twisted version of 
the action
of ${\rm GL}_2$ on $V_4\times V_1$, where $V_1$ is the standard 
representation of ${\rm GL}_2$ and $V_4={\rm Sym}^4(V_1)$.
The invariant theory used is that of covariants (or more
precisely, bi-covariants) for this action.
The generators of the ring of bi-covariants are known classically.
The construction of modular forms is realized by substituting 
the coordinates of the basic forms $\chi_{4,-2}$ and $E_{1,1}$
in the covariants. In general a covariant yields a meromorphic
modular form with possible poles only along the curve $T_1$
where the scalar-valued form $\zeta$ vanishes. This curve $T_1$ is
the locus where
the Jacobian of our genus $3$ curve is a product of 
an abelian surface and a fixed elliptic curve with multiplication
by third roots of unity.

In order to apply this effectively we need to construct explicitly
Fourier-Jacobi expansions of the generating modular forms
$\zeta$, $\chi_{4,-2}$ and $E_{1,1}$. We use gradients of
theta functions to construct these basic forms. 

To check
holomorphicity of the modular forms obtained from covariants
we need also the Taylor expansions of these generating forms
along the modular curve $T_1$ on our Picard
modular surface.

As an application we show how to construct the generators of
rings of scalar-valued modular forms and of modules of
vector-valued modular forms from invariants and covariants.
In particular, we determine generators of modules of vector-valued
Picard modular forms of weight $(4,k)$.

As a possible further application we mention
that the description of modular forms by covariants should allow
a description and construction of these Picard 
modular forms in positive characteristic. 

It is a great pleasure to dedicate this paper to Don Zagier
who through his work and in his contacts with us has been a source
of inspiration for both of us.
\end{section}
\section*{Acknowledgements}
The first author was supported by Simons Foundation Award 
546235 at the Institute for Computational and
Experimental Research in Mathematics at Brown University.
He thanks Loughborough University for hospitality during
the preparation of this paper.
The second author thanks the Mittag-Leffler Institute for
excellent working conditions during the final preparations
of this paper. He also thanks YMSC of Tsinghua University
and the University of Luxembourg for support.

\medskip

\tableofcontents
\begin{section}{Picard modular forms}\label{sectionPMF}
We briefly recall the notion of Picard modular forms on the
$2$-ball. We refer to \cite{Fi,C-vdG} for more details.
Let $F={\QQ}(\sqrt{-3})$ with ring of integers $O_F={\ZZ}[\rho]$
for a primitive third root of unity $\rho$ 
and units $O_F^{\times}=\mu_6$. 
We consider the 
non-degenerate Hermitian form $h$ of signature $(2,1)$ 
on the $F$-vector space $Z=F^3$ given by
$$
z_1z_2'+z_1'z_2+z_3z_3'\, ,
$$
where the prime indicates the Galois conjugate. It defines
an algebraic group $G$
over ${\QQ}$ consisting of the similitudes of $h$
$$
G({\QQ})=\{ g \in {\rm GL}(3,F): h(gz)=\eta(g)h(z)\}
$$
with multiplyer homomorphism $\eta: G \to {\GG}_m$. This is a
group of type ${\rm GU}(2,1,F)$. 
 We let $G^0=\ker(\eta)$. The two arithmetic groups of interest are
$$
\Gamma=G^0({\ZZ}), \quad \Gamma_1=G^0({\ZZ})\cap \ker \det \, .
$$ 
After choosing an embedding $F\hookrightarrow {\CC}$ we can identify 
$F\otimes_{\QQ}{\RR}$ with ${\CC}$ and $G({\RR})$ acts on
the complex vector space $Z_{\RR}=Z\otimes_{\QQ}{\RR}$ via the
standard representation. An element $g$ of $G^{+}({\RR})=
\{ g \in G({\RR}): \eta(g)>0\}$ preserves the set of
negative complex lines
$$
\mathfrak{B}=\{ L: L\subset Z \otimes_{\QQ}{\RR}, \dim_{\CC}L=1, h_{|L} < 0\} \, .
$$
The action can be given explicitly by first identifying $\mathfrak{B}$
via $u=z_3/z_2$ and $v=z_1/z_2$ with a complex $2$-ball
$$
\mathfrak{B}=\{ (u,v) \in {\CC}^2: v+\bar{v}+u\bar{u} <0 \}\, .
$$
Then an element $g=(g_{ij})$ acts by 
$$
g\cdot (u,v)=
\left(\frac{g_{31}v+g_{32}+g_{33}u}{g_{21}v+g_{22}+g_{23}u},\frac{g_{11}v+g_{12}+g_{13}u}{g_{21}v+g_{22}+g_{23}u}\right) \, .
$$
The quotient $X_{\Gamma}=\Gamma\backslash \mathfrak{B}$ is called
a Picard modular surface.
It is not compact, but can be compactified by adding one cusp.
It was studied in detail by Holzapfel and Feustel, see \cite{H,Fe}. 
The two congruence subgroups
$$
\Gamma[\sqrt{-3}]=
\{ \gamma \in \Gamma : \gamma \equiv 1_3 \, (\bmod \sqrt{-3})\}
\quad {\rm and} \quad 
\Gamma_1[\sqrt{-3}]=\Gamma[\sqrt{-3}] \cap\Gamma_1
$$
will also play a role here.
For later use we record the following lemma, see \cite[p.\ 329] {Shiga}.

\begin{lemma}\label{generators}
The following six elements generate the group $\Gamma[\sqrt{-3}]$:
$g_0=\rho \, 1_3$ and
$$
g_1=\left(\begin{smallmatrix} 1 & 0 & 0 \\ 0 & 1 & 0 \\ 0 & 0 & \rho\end{smallmatrix}\right), 
g_2=\left(\begin{smallmatrix} 1 & 0 & 0 \\ \sqrt{-3} & 1 & 0 \\ 0 & 0 & 1\end{smallmatrix}\right), 
g_3=\left(\begin{smallmatrix} 1 & 0 & 0 \\ \rho-1 & 1 & \rho-1 \\ 1-\rho^2 & 0 & 1\end{smallmatrix}\right), 
g_4=\left(\begin{smallmatrix} 1 & \sqrt{-3}  & 0 \\ 0 & 1 & 0 \\ 0 & 0 & 1\end{smallmatrix}\right),
g_5=\left(\begin{smallmatrix} 1 & \rho-1  & \rho-1 \\ 0 & 1 & 0 \\ 0 & 1-\rho^2 & 1\end{smallmatrix}\right).
$$
\end{lemma}

The quotient $X_{\Gamma_1[\sqrt{-3}]}=\Gamma_1[\sqrt{-3}]\backslash \mathfrak{B}$ can be compactified by 
adding four cusps represented by
$[1:0:0]$, $[0:1:0]$ $[\rho:1:1]$ and $[\rho:1:-1]$. 
We have an isomorphism
$$
\Gamma/\Gamma_1[\sqrt{-3}] \cong \mathfrak{S}_4 \times \mu_6,
\quad g \mapsto (\sigma(g),\det(g))\, ,
$$
where $\mathfrak{S}_4$ is the symmetric group and
 $\sigma(g)$ is the permutation of the four cusps.
The $\mu_6$--part is generated by $-1_3$ and $g_1=\text{diag}(1,1,\rho)$,
while the $\mathfrak{S}_4$--part is generated by:
$$
r_1=\left(\begin{matrix}0&-1&0\\-1&0&0\\0&0&-1\end{matrix}\right)\, ,
\quad
r_2=\left(\begin{matrix}-1&0&0\\0&-1&0\\0&0&1\end{matrix}\right)
\quad
{\rm{and}}
\quad
r_3=\left(\begin{matrix}1&\rho^2&1\\0&1&0\\0&-1&1\end{matrix}\right).
$$
Note that $\Gamma_1 / \Gamma_1[\sqrt{-3}] \simeq \mathfrak{S}_4$ and
the three elements $r_1,r_2$ and $r_3$ correspond to the permutations
$(12), (34)$ and $(234)$ in $\mathfrak{S}_4$.

The action of $G^{+}({\RR})$ on $\mathfrak{B}$ 
defines two factors of automorphy:
$$
j_1(g,u,v)=g_{21}v+g_{22}+g_{23}u\, ,\quad
j_2(g,u,v)=
\det(g)^{-1}
\left(
\begin{matrix}
G_{32}u+G_{33} & G_{32}v+G_{31}\\
G_{12}u+G_{13} & G_{12}v+G_{11}
\end{matrix}
\right) \, ,
$$
where $G_{ij}$ denotes the minor of $g_{ij}$.
We have
$$
\det(j_2(g,u,v))=j_1(g,u,v)/\det(g) \, .
$$
The factor of automorphy $j_2$ agrees with 
the canonical  factor  of automorphy as defined by Satake,
see \cite[Chapter II.5]{Satake}; see also \cite{Shimura1978}.

For a pair $(j,k)$ of integers and $g\in G^+({\RR})$,  we define
a slash operator on functions $f: \mathfrak{B} \to \Sym^j(\CC^2)$,
\[
(f\vert_{j,k} g)(u,v)=
j_1(g,u,v)^{-k}\Sym^{j}(j_2(g,u,v)^{-1})f(g\cdot(u,v)).
\]
For a discrete subgroup $\Gamma'$ of $G^+({\RR})\cap \ker \eta$ 
and a character $\chi$ of $\Gamma'$ of finite order, we
define the space of modular forms of weight $(j,k)$ 
and character $\chi$ on $\Gamma'$ as
\[
M_{j,k}(\Gamma',\chi)=
\left\{
f:B \to \Sym^j(\CC^2)\, | \, f \, \text{holomorphic}, 
f\vert_{j,k} g=\chi(g)\, f\,  \text{for any}\,  g\in \Gamma'
\right\}.
\]
We denote by $S_{j,k}(\Gamma',\chi)$ the subspace of cusp forms 
of $M_{j,k}(\Gamma',\chi)$.
For $j=0$, that is, for scalar-valued forms, 
we shorten these notations by just 
writing $M_{k}(\Gamma',\chi)$ and $S_{k}(\Gamma',\chi)$,
and $M_{k}(\Gamma')$ and $S_k(\Gamma')$ if $\chi$ is trivial.
We have the graded ring of modular forms on $\Gamma'$  with
$$
M(\Gamma')=\bigoplus_{k\geqslant 0}M_k(\Gamma')\, .
$$

We apply this to the case where
$\Gamma'$ equals to one of the groups
$\Gamma, \Gamma_1, \Gamma[\sqrt{-3}]$ and $\Gamma_1[\sqrt{-3}]$.

\begin{remark} 
The isomorphism $\Gamma[\sqrt{-3}]/\Gamma_1[\sqrt{-3}] \cong 
\mu_3$ via $g \mapsto \det(g)$
gives a decomposition
$$
M_{j,k}(\Gamma_1[\sqrt{-3}])=
\oplus_{l=0}^2 M_{j,k}(\Gamma[\sqrt{-3}], {\det}^{l})\, ,
$$
and similarly,  $\Gamma/\Gamma_1\cong \mu_6$ 
via $g\mapsto \det(g)$ gives
$$
M_{j,k}(\Gamma_1)= \oplus_{l=0}^5 
M_{j,k}(\Gamma, {\det}^{l})\, .
$$
However, since $-1_3 \in \Gamma$
acts by $(-1)^{j+k}$ on $M_{j,k}(\Gamma_1)$ here we may restrict
$l$ by $j+k\equiv l (\bmod 2)$, that is, $l\in \{0,2,4\}$ or $l\in \{1,3,5\}$.  
But note that if we view a modular form on $\Gamma_1$ 
as a modular form on $\Gamma_1[\sqrt{-3}]$ the notation of the
character may change since $\Gamma/\Gamma_1 \cong \mu_6$, but 
$\Gamma[\sqrt{-3}]/\Gamma_1[\sqrt{-3}]\cong \mu_3$.
\end{remark}

We have two order $2$ characters on $\Gamma$. The first one is
${\det}^3$, while the second one, denoted $\epsilon$, comes from 
the isomorphism 
$\Gamma/\Gamma_1[\sqrt{-3}]\cong \mathfrak{S}_4 \times \mu_6$ 
and the map $\mathfrak{S}_4 \times \mu_6 \to \{ \pm 1 \}$
given by $(\sigma,z) \mapsto {\rm sgn}(\sigma)$ with ${\rm sgn}$
the sign character on $\mathfrak{S}_4$. 

The isomorphism $\Gamma_1/\Gamma_1[\sqrt{-3}] \cong \mathfrak{S}_4$
makes $M_{j,k}(\Gamma_1[\sqrt{-3}])$ into 
a representation of $\mathfrak{S}_4$ and we have
$$
M_{j,k}(\Gamma_1,\epsilon)=M_{j,k}(\Gamma_1[\sqrt{-3}])^{s[1^4]},
$$
with $s[1^4]$ the alternating 
$\mathfrak{S}_4$-representation and with $M_{j,k}(\Gamma_1[\sqrt{-3}])^{s[1^4]}$
denoting the subspace of $M_{j,k}(\Gamma_1[\sqrt{-3}])$ where $\mathfrak{S}_4$
acts via the alternating character. 

By \cite{H} the Baily-Borel compactification $X^*_{\Gamma[\sqrt{-3}]}$
of $X_{\Gamma[\sqrt{-3}]}= \Gamma[\sqrt{-3}]\backslash \mathfrak{B}$
can be identified with ${\PP}^2 \subset {\PP}^3$ given by the
hyperplane $x_1+x_2+x_3+x_4=0$ with the action of 
$\Gamma/\Gamma[\sqrt{-3}] \cong \mathfrak{S}_4 \times \mu_2$
given by $x_i \mapsto {\rm sgn}(\sigma)\, x_{\sigma(i)}$ and $\mu_2$ 
acting trivially.
Moreover, $X^*_{\Gamma_1[\sqrt{-3}]}$ can be identified
with the $3$-fold cover given by $\zeta^3= 
\prod_{1 \leq i < j \leq 4} (x_i-x_j)$.

The factor of automorphy $j_1$
corresponds to an orbifold line bundle
$L$ on $\Gamma\backslash \mathfrak{B}$
and the factor $j_2$ to a rank $2$ orbifold vector bundle $U$.
If we define $j_3=\det(g)$ to be the third factor of automorphy we have
$ \det(j_2)=j_1/j_3$.
This factor $j_3$ corresponds to $R={\det}(U)^{-1}\otimes L$, see \cite{B-vdG}.
Note that $R$ is a torsion line bundle.

When we speak of weight $(j,k,l)$ we refer to the factor of automorphy
$j_1^k\, {\rm Sym}^j(j_2)\, j_3^l$.

The $M(\Gamma)$-module
$$
{\MM}(\Gamma)=\bigoplus_{j,k \in {\ZZ}_{\geq 0}} M_{j,k}(\Gamma)
$$
can be made into a ring; indeed a modular form on $\Gamma$
of weight $(j,k,l)$
is a section of ${\rm Sym}^j(U) \otimes L^k \otimes R^{l}$
on $\Gamma \backslash \mathfrak{B}$ and
the canonical projection
${\rm Sym}^a(U) \otimes {\rm Sym}^b(U) \to  {\rm Sym}^{a+b}(U)$
and the usual multiplication of line bundles
determines the ring structure.
Similarly, we have a ring structure
on the $M(\Gamma_1)$-module ${\MM}(\Gamma_1)=\oplus M_{j,k}(\Gamma_1) $.

We now briefly summarize 
what is known about Picard modular forms on the groups
in question. Shiga studied in the sixties 
Picard modular forms using theta functions
in \cite{Shiga}. In the eighties Feustel and Holzapfel determined 
a few rings of scalar-valued modular forms, see below.
In the nineties Finis constructed a number of scalar-valued
Hecke eigenforms of small
weight and determined Hecke eigenvalues in \cite{Fi}.
Shintani discussed the notion of vector-valued modular forms
in \cite{Shintani}.

Bergstr\"om and
one of us studied in \cite{B-vdG} the cohomology of local systems on the arithmetic
quotient $\Gamma_1[\sqrt{-3}]\backslash \mathfrak{B}$ and gave dimension
formulas for the spaces $S_{j,k,l}(\Gamma[\sqrt{-3}])$. 
The interpretation of the Picard modular surface as a 
moduli of curves was used there to determine
experimentally by counting points over finite fields Hecke eigenvalues
of Picard modular forms on
$\Gamma_1[\sqrt{-3}]$.
Motivated by the early 
experimental results of \cite{B-vdG} we constructed
in \cite{C-vdG} a number of vector-valued modular forms
and determined the structure of a few modules of vector-valued modular forms.

\bigskip
We finish this section by recalling 
some results of Feustel and Holzapfel, \cite{Fe,H} 
on the structure of some graded
rings of scalar-valued modular forms. 
There exist modular forms  $\varphi_i \in M_3(\Gamma[\sqrt{-3}])$ for $i=0,1,2$,
and $\zeta \in S_6(\Gamma[\sqrt{-3}],\det{})$ such that
$$
M(\Gamma[\sqrt{-3}])=\CC[\varphi_0,\varphi_1,\varphi_2],
\quad \text{\rm and} \quad
M(\Gamma_1[\sqrt{-3}])=\CC[\varphi_0,\varphi_1,\varphi_2,\zeta]/(R),
$$
where $(R)$ is the ideal generated by the relation
$$
\zeta^3=-\frac{\rho}{3^7\sqrt{-3}}\varphi_0\varphi_1\varphi_2(\varphi_1-\varphi_0)(\varphi_2-\varphi_0)(\varphi_2-\varphi_1)\, .
\eqno(1)$$
The constant $-\rho/3^7\sqrt{-3}$ is due to our normalizations, see later.
The $\varphi_i$ are related with the coordinates $x_i$ of the Baily-Borel
compactification $X_{\Gamma[\sqrt{-3}]}^{*}$ via
$$
x_1=\varphi_0+\varphi_1+\varphi_2, \quad
x_2=-3\, \varphi_0+\varphi_1+\varphi_2, \quad
x_3=\varphi_0-3\, \varphi_1+\varphi_2, \quad
x_4=\varphi_0+\varphi_1-3\, \varphi_2, 
$$
and the action of $\mathfrak{S}_4$ by 
$x_i \mapsto {\rm sgn}(\sigma)x_{\sigma(i)}$
makes  $M_3(\Gamma[\sqrt{-3}])$ into the 
$\mathfrak{S}_4$-representation $s[2,1^2]$
corresponding to the partition $(2,1,1)$ of $4$.

The form $\zeta \in S_6(\Gamma[\sqrt{-3}],{\det})$ is 
$\mathfrak{S}_4$-anti-invariant. We thus can view $\zeta$ as an
element of $S_6(\Gamma_1,\epsilon)$.

One defines Eisenstein series $E_i$ of weight $i$ 
on the group $\Gamma$ or a smaller group by
\begin{align*}
E_6&=\varphi_0^2+\varphi_1^2+\varphi_2^2-\frac{2}{3}(\varphi_0\varphi_1+\varphi_0\varphi_2+\varphi_1\varphi_2)\in M_6(\Gamma),\\
E_9&=(-\varphi_0+\varphi_1+\varphi_2)(\varphi_0-\varphi_1+\varphi_2)(\varphi_0+\varphi_1-\varphi_2)\in M_9(\Gamma[\sqrt{-3}]) 
\cap M_9(\Gamma_1,\epsilon),\\
E_{12}&=-\frac{1}{3}(\varphi_0+\varphi_1+\varphi_2)(-3\varphi_0+\varphi_1+\varphi_2)(\varphi_0-3\varphi_1+\varphi_2)(\varphi_0+\varphi_1-3\varphi_2)\in M_{12}(\Gamma).
\end{align*}
Then we can describe the rings of modular forms on $\Gamma$ and $\Gamma_1$: 
$$M(\Gamma)=\CC[E_6,E_{12},E_9^2] \, ,
\quad \text{\rm and} \quad
M(\Gamma_1)=\CC[E_6,E_{12},E_9^2, \zeta E_9, \zeta^2]/(R_1),
$$
with the ideal $(R_1)$ generated by the relation 
$(E_9\, \zeta)^2=E_9^2 \zeta^2$ and
$$
\rho \, 2^{16} 3^{12} \zeta^6=
9\, E_6^4E_{12} -8 \, E_6^3 E_{12}^2+6\, E_6^2E^2_{12} 
-24\,  E_6 E_9^2 E_{12} 
+16\, E_9^4+E_{12}^3\, . \eqno(2)
$$
We also have
$$
M(\Gamma_1,\epsilon)={\CC}[E_6,E_9,E_{12},\zeta^2]/(R')
$$
with $(R')$ generated by the relation (2).

\end{section}
\begin{section}{A modular embedding}\label{embedding}
The arithmetic quotient $X_{\Gamma}=\Gamma\backslash \mathfrak{B}$
parametrizes principally polarized 
abelian threefolds with multiplication by $O_F$. Therefore there is a
morphism $X_{\Gamma} \to \mathcal{A}_3({\CC})$. 
We now describe the corresponding modular embedding
$\Gamma\backslash \mathfrak{B} \to 
{\rm Sp}(6,{\ZZ})\backslash \mathfrak{H}_3$ with $\mathfrak{H}_3$ the
Siegel upper half space of degree $3$
$$
\mathfrak{H}_3=
\{\tau \in {\rm Mat}(3,{\CC}): \tau^t=\tau, {\rm Im}(\tau)>0\}\, .
$$
Such modular embeddings were considered by Picard, Shiga and Holzapfel, see
\cite{P1,Shiga,H2}.

The lattice $O_F^3$ with Hermitian form
$h=z_1z_2'+z_2z_1'+z_3z_3'$
determines an alternating form $(2/\sqrt{3}){\rm Im}(h)$
and by taking as
${\ZZ}$-basis of this lattice
$$e_1=(\rho^2,0,0), \, e_2=(0,\rho^2,0), \, e_3=(0,0,\rho^2), \quad
f_1= (0,\rho,0), \, f_2=(\rho,0,0), \, f_3= (0,0,\rho)
$$
we can identify it with the standard symplectic lattice generated by
$e_1,e_2,e_3,f_1,f_2,f_3$ 
with $\langle e_i,e_j\rangle=0$, $\langle f_i,f_j\rangle=0$,
$\langle e_i,f_j\rangle=\delta_{ij}$. Here $\langle \, , \, \rangle$
denotes the alternating form that is the imaginary part of the Hermitian form.
 This defines an embedding
$\Gamma \to {\rm Sp}(6,{\ZZ})$.

If we take instead the symplectic basis
$(e_1,e_3,-f_2,f_1,f_3,e_2)$
we get the following modular embedding
$$
\iota:\mathfrak{B} \to \mathfrak{H}_3, \quad 
\sigma:\Gamma \to {\rm Sp}(6,{\ZZ})
$$ 
given by
\[
\iota(u,v)=
 \left(\begin{matrix} \frac{u^2+2\rho^2v}{1-\rho} & \rho^2u & \frac{\rho u^2-\rho^2v}{1-\rho}  \\
\rho^2u  & -\rho^2 & u \\
\frac{\rho u^2-\rho^2v}{1-\rho} & u & \frac{\rho^2u^2+2\rho^2v}{1-\rho}
 \end{matrix}\right)
\]
and for $g=(a_{ij}+\rho \, b_{ij})$
\[
\sigma(g)=
\left(\begin{matrix}
a_{11}-b_{11} & a_{13}-b_{13}  & -b_{11} & b_{12} & b_{13} & a_{12}-b_{12}  \\
a_{31}-b_{31} & a_{33}-b_{33}  & -b_{31} & b_{32} & b_{33} & a_{32}-b_{32} \\
b_{11} & b_{13}  & a_{11} & -a_{12} & -a_{13} & b_{12} \\
-b_{21} & -b_{23}  & -a_{21} & a_{22} & a_{23} & -b_{22} \\
-b_{31} & -b_{33}  & -a_{31} & a_{32} & a_{33} & -b_{32} \\
a_{21}-b_{21} & a_{23}-b_{23}  & -b_{21} & b_{22} & b_{23} & a_{22}-b_{22}
\end{matrix}\right)
\]
The pullback of the stabilizer in ${\rm Sp}(6,{\ZZ})$ of $\iota(\mathfrak{B})$
is the group $\Gamma$. This can be derived from the Torelli theorem applied to
curves of genus $3$ that are triple cyclic covers of ${\PP}^1$.

\smallskip

Let ${\EE}$ be the Hodge bundle on $\mathcal{A}_3({\CC})$, that is,
the cotangent bundle of the universal abelian threefold along the zero section.
Via the map $\iota$ we can pull back ${\EE}$ to $\Gamma\backslash \mathfrak{B}$.

We wish to express the pull back of the Hodge bundle in terms of the
automorphic bundles $L$, $U$ and $R$ 
associated to  the factors of automorphy $j_1$, 
$j_2$ and $j_3$. 

\begin{lemma}\label{pullbacksplits}
The pullback of the Hodge bundle ${\EE}$ over $\mathcal{A}_3({\CC})$ to $\Gamma_1\backslash \mathfrak{B}$ 
is isomorphic to $U\oplus L$. The pullback of $\det({\EE})$ is $L^2\otimes R^{-1}$.
\end{lemma}
\begin{proof}
The second statement follows from the first when one uses 
$\det(U)=L\otimes R^{-1}$. In order to prove the first one,
we observe that the Hodge bundle corresponds to the factor of
automorphy $(c\tau +d)$ for ${\rm Sp}(6,{\RR})$ acting on $\mathfrak{H}_3$. 
With $\tau= \iota(u,v)$ and $\sigma(g)=(a,b;c,d)$ 
we find for diagonal matrices 
$g={\rm diag}(g_1,g_2,g_3)$ with $g_i=a_i+\rho b_i$ that
$c\,  \iota(u,v)+d$ equals
$$
\left( \begin{matrix} a_2 & 0 & -b_2 \\ 
-b_3 \rho^2 u & \bar{g}_3 & -b_3 u \\
b_2 & 0 & -b_2 \\
\end{matrix} \right)
$$
with characteristic polynomial $(X-g_2)(X-\bar{g}_3)(X-\bar{g}_2)$.
Since $g$ respects the Hermitian form, we have $g_1\bar{g}_2=g_3\bar{g}_3=1$,
and therefore $j_2(g,(u,v))={\rm diag}(\bar{g}_3, \bar{g}_2)$
and $j_1(g,(u,v))=g_2$. 
Hence, up to a base change we have
$c\tau+d=j_2(g,(u,v))\oplus j_1(g,(u,v))$.
Since arbitrary Hermitian matrices can be
diagonalized the lemma follows.
\end{proof}

\end{section}
\begin{section}{Modular curves}\label{ModularCurves}
Picard modular surfaces contain modular curves defined by
positive vectors in the lattice $O_F^3$. Though these curves
were considered by Feustel, Kudla, Cogdell and others, their
geometry on these surfaces did not yet get the attention that their
counterparts on Hilbert modular surfaces got.  Here we need just
two curves that play a role. 

A vector $w=(a,b,c) \in O_F^3$ with positive norm
$ab'+a'b+cc'$ defines a $1$-ball $\mathfrak{B}_w$ 
inside $\mathfrak{B}=
\{ L: L \subset Z_{\RR}: \dim_{\CC}L=1, h_{|L}<0\}$
by the condition $L\bot w$, or equivalently by
$$
a'  + b'v + c' u=0 \, . \eqno(3)
$$
This defines a curve in $\Gamma\backslash \mathfrak{B}$
and also in the Baily-Borel compactification $X^{*}_{\Gamma}$.
We can define a  modular curve $T_N$ in 
$\Gamma\backslash \mathfrak{B}$ as the union of all curves
defined by equations (3) with $ab'+a'b+cc'=N$. 
Its closure in $X^{*}_{\Gamma}$ is also denoted $T_N$.

In the following
we need the two curves $T_1$  and $T_2$.
The curve $T_2$ was studied in \cite{P-S}. 

The curve $T_1$ has one irreducible component on $X^{*}_{\Gamma}$ as one sees by verifying that the action of $\Gamma$ on positive vectors $(a,b,c)$ 
with $ab'+a'b+cc'=1$ is transitive 
using the generators of $\Gamma$, see Section \ref{sectionPMF}. 
It can be defined by $u=0$ and viewed  
as a quotient of the upper half plane  
$\mathfrak{H}$ embedded in $\mathfrak{B}$ by
$\tau \mapsto (0,\sqrt{-3}\tau)$.  The image in
$\Gamma_1[\sqrt{-3}]\backslash \mathfrak{B}$  is isomorphic to
$\Gamma_0(3)\backslash \mathfrak{H}$ with $\Gamma_0(3)$ the usual
congruence subgroup of ${\rm SL}(2,{\ZZ})$.
The modular form $\zeta$ vanishes
on $T_1$ since $\zeta(-u,v)=-\zeta(u,v)$.

The curve $T_2$ is a Shimura curve associated to the unit group 
of a maximal order in the quaternion
algebra $\left( {-3, 2 \over {\QQ}}\right)$ of discriminant $6$. 
The curve $T_2$ has one irreducible component on $X^{*}_{\Gamma}$ 
and it can be defined by $v=-1$ and 
is the fixed point locus of the involution
$$
\xi: (u,v) \mapsto (-u/v,1/v)
$$
that is induced by the symmetry
$(z_1,z_2,z_3) \mapsto (z_2,z_1,-z_3)$ of our Hermitian
form $z_1z_2'+z_1'z_2+z_3z_3'$. The involution $\xi$
induces an action on spaces of modular forms. The action
on a modular form
$f \in M_k(\Gamma[\sqrt{-3}])$ restricted to $v=-1$ is by 
multiplication by $(-1)^k$. 
In particular, the  Eisenstein series 
$E_9$ vanishes on the fixed point locus of $\xi$.

More precisely, on $X^*_{\Gamma}$ the modular forms $\zeta^6$ and $E_9^2$ 
give rise to the cycle relations:
$$
6\, \lambda_1= [T_1], \qquad 9\, \lambda_1= [T_2]\, ,
$$
where $\lambda_1$ represents the first Chern class of $L$
and the classes $[T_1]$ and $[T_2]$ are Q-classes on the orbifold
$X_{\Gamma}^*$ in the sense of Mumford \cite{Mumford1983}. Indeed,
the modular forms $\zeta^6$ and $E_9^2$ that live on $X_{\Gamma}$
have divisors $6\, T_1$ and $2\, T_2$, where the multiplicities
come from the fact that a generic point of $T_1$ (resp.\ $T_2$) 
has a stabilizer of order $6$ (resp.\ of order $2$).  Equivalently,
one can also work
on $X^*_{\Gamma_1[\sqrt{-3}]}$ where one has the modular forms
$\zeta$ and $E_9$ with divisors $T_1$ and $T_2$; see for example
the Taylor expansion of $\zeta$ along $u=0$ in Section \ref{Restriction-section}. 
The volume form on the orbifold defines a class $T_0$, see \cite{Cogdell}.

\begin{corollary}\label{TNseries}
If $[T_N]$ denotes for $N\in {\ZZ}_{\geq 0}$ 
the $Q$-class of the curve $T_N$ on $X^*_{\Gamma}$,
then the series $\sum_{N=0}^{\infty} [T_N]\,  q^N$ equals $F\otimes \lambda_1$ with 
$$
F=-1/6 + 6\, q + 9\, q^2 + 42\, q^3 + 78\, q^4 + O(q^5)
$$
a modular form in $M_3(\Gamma_0(3),\big( \frac{\cdot}{3}\big))$.
\end{corollary}
\begin{proof} We can work on the minimal resolution of singularities 
$\tilde{X}_{\Gamma_1[\sqrt{-3}]}$ of $X_{\Gamma_1[\sqrt{-3}]}^*$
and consider the classes $[T_N^c]$ there that are defined by a linear
combination (with ${\QQ}$-coefficients) of 
$T_N$ plus a sum of resolution curves such that $T_N^c$ is orthogonal
to the cusp resolutions. Then we can use the result of
Cogdell \cite[Thm.\ on page 126]{Cogdell}, the analogue 
for Picard modular surfaces of the Hirzebruch-Zagier theorem 
on curves on Hilbert modular surfaces. It says that $\sum_N [T_N^c] q^N$
is a modular form of weight $3$ on $\Gamma_0(3)$ with Dirichlet character.
Since $\dim M_3(\Gamma_0(3), \big( \frac{}{3}\big))=2$ and we know the 
coefficients of $q$ and $q^2$, this identifies the
modular form.
\end{proof}

\smallskip

In the Baily-Borel 
compactification 
$X^*_{\Gamma[\sqrt{-3}]}$, identified with ${\PP}^2$
viewed as the hyperplane $x_1+x_2+x_3+x_4=0$ in ${\PP}^3$
and with the action of $\mathfrak{S}_4$ given by 
$x_i\mapsto {\rm sgn}(\sigma)x_{\sigma(i)}$, the 
lines $x_i=x_j$ describe the six components of $T_1$. 
Similarly, the curve $T_2$
has three components and is given by 
$x_i+x_j=0$ for $1\leq i< j\leq 4$.

We now describe the image of $T_1$ 
under the modular embedding $\iota$ constructed
in the preceding section. The image in $\mathfrak{H}_3$ under $\iota$ 
of the curve given by $u=0$ is
$$
\left\{ \left( 
\begin{matrix}
\tau_{11} & 0 & \tau_{12} \\
0 & 1+\rho & 0 \\
\tau_{12} & 0 & \tau_{22} \\
\end{matrix}\right)
\in \mathfrak{H}_3 : \tau_{11}=\tau_{22}= -2\tau_{12}
\right\} \, .
$$
In particular, 
$\iota(T_1)\subset \mathcal{A}_{2,1} \subset \mathcal{A}_3$, with
$\mathcal{A}_{2,1}$ the moduli of abelian varieties that are
products.
The equations $\tau_{11}=\tau_{22}$ and $\tau_{11}+2\tau_{12}=0$
define two Humbert surfaces of discriminant $4$ in $\mathcal{A}_2$,
cf.\ \cite[p. 210]{vdG-HMS}.

The fact that $\iota(T_1)$ is contained in  $\mathcal{A}_{2,1}$
means that an abelian threefold $X$ representing a point of $T_1$
splits as a product $X=X_2\times X_1$ with $X_2$ a principally polarized abelian surface
and $X_1$ an elliptic curve. Since $X_1$ has multiplication by $\rho$
the curve $X_1$ is rigid. This means that the Hodge bundle ${\EE}$
restricted to an irreducible component of $T_1$ on $X_{\Gamma_1[\sqrt{-3}]}$ 
has a trivial factor. By Lemma \ref{pullbacksplits} the Hodge bundle ${\EE}$
splits as $U\oplus L$ and since the action of $\rho$  on the $2$-dimensional
factor $X_2$ has eigenvalues $(\rho,\rho^2)$ (see \cite[Section 5.5]{B-vdG}),
we see that this constant factor is contained in $U$.

This means that the bundle $U$
restricted to an irreducible component $T$ of our modular curve 
$T_1$ on $X_{\Gamma_1[\sqrt{-3}]}$ 
 is of the form $\mathcal{O}_{T} \oplus N$
with $N$ the line bundle obtained by the 
restriction of $\det(U)$; its sections
 correspond to modular forms of weight $1$. 

However, the curve $T_1$ 
on $\Gamma_1[\sqrt{-3}]\backslash \mathfrak{B}$ 
is reducible with six smooth irreducible components 
meeting in ordinary double points.

\begin{lemma} \label{taylor}
Let $f$ be a meromorphic modular form of weight $(j,k)$ 
on $\Gamma_1[\sqrt{-3}]$ that is 
holomorphic outside the curve $T_1$. 
If $f$ has order $r$ along $T_1$, then
the first non-zero Taylor term of $f$ along $T_1$ is an element of 
$$
\oplus_{i=0}^j M_{i+k+r}^{(r)}(\Gamma_1(3)) \, ,
$$ 
with $M_k^{(s)}(\Gamma_1(3))$ 
the space of meromorphic modular 
forms of weight $k$ on $\Gamma_1(3)$ 
that are holomorphic outside the orbit of $\tau_0=(1-\rho^2)/3 \in \mathfrak{H}$ and have
order at least $s$ at $\tau_0$.
\end{lemma} 
\begin{proof}
Restricting the vector bundle ${\rm Sym}^j(U) \otimes L^k$
to an irreducible component $T$ of the modular curve $T_1$ gives the vector bundle
$B=\oplus_{i=0}^j N^{\otimes (i+k)}$.  
Moreover, the conormal space of the component $T$ of $T_1$
in $\Gamma_1[\sqrt{-3}]\backslash \mathfrak{B}$ when pulled back to
$\mathfrak{H}$ can be identified with a fibre of the line bundle $N$,
as one sees by looking at the action of $\rho$ on 
the deformation space of an abelian threefold $X=X_2\times X_1$ representing
a point of $T$. Thus the conormal bundle of $T$ can be identified with $N$.
The $r$th term in the Taylor expansion along $T$ of $f$,
viewed as a section of ${\rm Sym}^j(U)\otimes L^k$,
is a section of $B\otimes N^{\otimes r}$.
Correcting for the double point of $T_1$ lying on a component
$T$, that is represented by $\tau_0$, implies the result.
\end{proof}

\smallskip

For later use we discuss the Taylor development of modular forms
along the curve $T_1$. Recall that this curve is represented by $u=0$
in $\mathfrak{B}$.
We can apply Proposition 8.4 of \cite{C-vdG} that we recall for convenience:
Let $f\in M_{j,k}(\Gamma[\sqrt{-3}],\text{det}^l)$ and write
\[
f(u,\sqrt{-3}\tau)=\sum_{n\geqslant 0}
\left[
\begin{smallmatrix}
f_n^{(0)}(\tau) \\ \vdots \\ f_n^{(j)}(\tau) 
\end{smallmatrix}
\right]
u^n \, .
\]
We write $\Gamma(3)$ for the principal  congruence subgroup of ${\rm SL}(2,{\ZZ})$,
and $\Gamma_1(3)$, $\Gamma_0(3)$ for the usual congruence subgroups.

\begin{proposition}\label{quasi-modularity}
The first component $f_n^{(0)}$ is a modular form of weight $k+n$ on $\Gamma_1(3)$
and a cusp form if $n>0$. Moreover $f_n^{(m)}$ vanishes unless $n+j-m\equiv l \bmod 3$.
The function $f_0^{(m)}$  is a modular form of weight $k+m$ on $\Gamma_1(3)$, while
for $n>0$ the function $f_n^{(m)}$  is a quasi-modular form of weight $k+m+n$ on $\Gamma_1(3)$.
\end{proposition}

The proof was not given in \cite{C-vdG}. Since we use this proposition
and a variant later, we give some details. 
The modular embedding of $T_1$ on $X_{\Gamma[\sqrt{-3}]}$ is given by
$$
\left( \begin{matrix} a & b \\ c & d \\ \end{matrix} \right) \mapsto
\left( \begin{matrix} a & \sqrt{-3} b & 0 \\
c/\sqrt{-3} & d & 0 \\ 0 & 0 & 1\\ \end{matrix} \right),
\qquad \tau \mapsto (0,\sqrt{-3}\, \tau) \, .
$$
We write $F\in M_{j,k}(\Gamma[\sqrt{-3}],{\det}^l)$ as
$$
F(u,v)=\left( \begin{matrix} F^{(0)}\\ \vdots \\ F^{(j)} \\
\end{matrix} \right)
\qquad \text{\rm with} \quad F^{(m)}= \sum_{n=0}^{\infty} F^{(m)}_n(v) u^n \, .
$$ 
Changing coordinates by setting $f^{(m)}_n(\tau)= F^{(m)}_n(\sqrt{-3}\tau)$,
the modularity of $F$ implies the following equation.
\begin{vergelijking}\label{vergelijkingen}
$$
\sum_{n=0}^{\infty} 
\left(\begin{matrix} f^{(0)}_n(\tau) \\ \vdots \\ f^{(j)}_n(\tau) \\
\end{matrix} \right)\, u^n =
(c\tau+d)^{-k-j}
{\rm Sym}^j\left( \begin{matrix} c\tau+d & 0 \\ -cu/\sqrt{-3} & 1 \\
\end{matrix} \right)
\sum_{n=0}^{\infty} (c\tau+d)^{-n} \left( \begin{matrix} 
 f^{(0)}(\frac{a\tau+b}{c\tau+d}) \\ \vdots \\
 f^{(j)}(\frac{a\tau+b}{c\tau+d})  \\
\end{matrix} \right) u^n
$$
\end{vergelijking}
Here the matrix ${\rm Sym}^j\left( \begin{matrix} c\tau+d & 0 \\ -cu/\sqrt{-3} & 1 \\
\end{matrix} \right)
$ is a lower diagonal matrix with entry on place $(r,s)$ for $r\geq s$
equal to
$$
\binom{j}{j+1-r} \left(\frac{-c}{\sqrt{-3}}\right)^{r-s} 
(c\tau+d)^{j+1-r}\, .
$$
From this it follows that $F^{(m)}_n$ is a modular form 
(and not only quasi-modular) 
if $F^{(\mu)}_{\nu} =0$ for all $\mu<m$ and $\nu < n$.

\bigskip

\noindent
\textbf{Modular forms on $\Gamma_1(3)$.}
For later use we recall some facts about elliptic modular forms
of level $3$.
Recall that the ring of modular forms on $\Gamma(3)$
equals
\[
M(\Gamma(3))=
\CC[\vartheta,\psi]\, , \quad \text{where} \quad
\vartheta(\tau)=\sum_{\alpha \in O_F}q^{N(\alpha)}
\quad \text{and} \quad
\psi(\tau)=\frac{\eta(3\tau)^3}{\eta(\tau)}
\]
with $\vartheta$ and $\psi$ of weight $1$ and $\eta(\tau)=q^{1/24}\prod_{n\geqslant 1}(1-q^n)$
the Dedekind eta-function, and for $\tau \in \mathfrak{H}$
we set as usual $q=e^{2 \pi i \tau}$.
Since $\Gamma_1(3)/\Gamma(3)$ is cyclic of order $3$ generated by
$
T=
\left(
\begin{smallmatrix}
1 & 1 \\ 0 & 1
\end{smallmatrix}
\right)
$
and $\psi\vert_{1} T=\rho\,\psi$, we get the structure of the ring of modular forms on $\Gamma_1(3)$
\[
M(\Gamma_1(3))=\CC[\vartheta,\psi^3].
\]
For example, the form $F$ of Corollary \ref{TNseries} is 
$(54\, \psi^3-\vartheta^3)/6$. 
To lighten notation we will sometimes use the relation
$$
\psi(\vartheta^3-\psi^3)=\eta^8\, .
$$
Note that 
$$
M_{2k}(\Gamma_1(3))=M_{2k}(\Gamma_0(3))\quad \text{\rm and} \quad 
M_{2k+1}(\Gamma_1(3))=M_{2k+1}(\Gamma_0(3),\big(\frac{\cdot}{3}\big))\, .
$$
By a result of Kaneko-Zagier (see \cite[Proposition 1, part b]{KZ}) we know that the graded ring $\widetilde{M}(\Gamma_1(3))$ of quasi-modular forms on
$\Gamma_1(3)$ is given by
\[
\widetilde{M}(\Gamma_1(3))=
M(\Gamma_1(3))\otimes \CC[e_2]\simeq \CC[\vartheta,\psi^3,e_2]\, ,
\]
where $e_2$ is the Eisenstein series of weight $2$ on ${\rm SL}(2,\ZZ)$.
We normalise $e_2$ such that its
Fourier expansion is given by
\[
e_2(\tau)=1-24\sum_{n\geqslant 1}\sigma_1(n)q^n.
\]
Examples of modular forms of level $3$ are given by
$\Theta_j(\tau)= \sum_{\alpha \in O_F}\alpha^jq^{N(\alpha)}\in M_{j+1}(\Gamma_1(3))$. 
Observe that  $\Theta_j$ is a cusp form as soon as $j>1$
and identically zero if $j \not \equiv 0 \bmod 6$.
For example, we have $\Theta_0=\vartheta$ and 
\[
\Theta_6=
6\vartheta \psi^3(\vartheta^3-27\psi^3)=
6\, \vartheta\psi^2\eta^8,
\quad
\Theta_{12}=
\vartheta \eta^{8} \left(\eta^{16}+18 \psi^{4} \eta^{8}+729 \psi^{8}\right) \, .
\]

Recall that $M_k^{(r)}(\Gamma_1(3))$ is the space of meromorphic
modular forms of weight $k$ that are holomorphic outside the orbit
of $\tau_0=(1-\rho^2)/3$ and have order at least $r$ at $\tau_0$.

The form $\vartheta$ is non-zero outside the orbit of $\tau_0$.
Multiplication by $\vartheta^{-r}$ provides an isomorphism for $r\in {\ZZ}$
$$
M_r^{(r)}(\Gamma_1(3)) \langepijl{\sim} M_0(\Gamma_1(3))={\CC} \, . \eqno(4)
$$
\end{section}
\begin{section}{A stack quotient}\label{StackQuotient}
In this section we discuss the moduli stack of curves of genus $3$
that are a cyclic cover of degree $3$ of the projective line.
We consider smooth projective curves $C$ over ${\CC}$ of genus $3$ together an 
automorphism $\alpha$ of order~$3$ such that the eigenvalues on 
$H^0(C,\Omega^1_C)$ are $\rho,\rho,\rho^2$. 
An isomorphism $(C,\alpha) \langepijl{} (C',\alpha')$ is an isomorphism
$\nu: C\to C'$ such that $\nu \alpha= \alpha' \nu$.
We let $\mathcal{N}$ denote the moduli stack over ${\CC}$ 
of such curves.

A choice $\omega_1,\omega_2$ of a basis of the $\rho$-eigenspace
$H^0(C,\Omega_C^1)^{\rho}$ defines a morphism of degree $3$ to
$C/\alpha={\PP}^1$. By the holomorphic Lefschetz fixed point 
formula the automorphism $\alpha$ has 
five fixed points on $C$, four of which have action by $\rho$ on the tangent
space and one with action by $\rho^2$. 

The $\rho$-eigenspace of ${\rm Sym}^4(H^0(C,\Omega^1_C))$ has dimension $7$,
whereas the $\rho$-eigenspace of 
$H^0(C,(\Omega_C^1)^{\otimes 4})$ has dimension $6$. (This follows
from the holomorphic Lefschetz formula; or by the simple argument that
the ternary quartic defining the canonical image of $C$ in ${\PP}^2$ must lie 
in an eigenspace and all elements with eigenvalue $1$ or $\rho^2$ are divisible by $\eta$,
a generator of $H^0(C,\Omega_C^1)^{\rho^2}$,
and this would give a reducible equation.)
After choosing a generator 
$\eta$ of $H^0(C,\Omega_C^1)^{\rho^2}$ we thus find a non-trivial relation
$$
b_1 \eta^3 \omega_1 + b_2 \eta^3 \omega_2= \sum_{i=0}^4 a_i \, \omega_1^{4-i}\omega_2
$$
with $b_i,a_j \in {\CC}$.
By setting $f_1=b_1x_1+b_2x_2$ and $f_4=\sum_{i=0}^4 a_i x_1^{4-i}x_2^i$ and
observing that $f_1$ is not identically zero, we obtain an
equation
$$
y^3 f_1=f_4 \, . \eqno(5)
$$
This represents the canonical image of $C$.
A different normalization 
is obtained by putting $y=\tilde{y}/f_1$ 
which gives $\tilde{y}^3= f_4f_1^2$.
By putting the zero of $f_1$ at infinity 
we find yet another normalization: 
an affine equation $u^3=f$ with $f$ of degree
$4$ in $v$. 
The map $\alpha$ corresponds to the field exension ${\CC}(u,v)/{\CC}(v)$.  

Changing the choice of basis of $H^0(C,\Omega_C^1)^{\rho}$
corresponds to an action of ${\rm GL}_2$. Changing the basis $\eta$ corresponds to an
action of ${\GG}_m$. Together this defines an action of the 
subgroup $\mathcal{G}={\rm GL}_1 \times {\rm GL}_2 \subset {\rm GL}_3$ 
on $H^0(C,\Omega^1_C)$
that preserves the decomposition in eigenspaces for $\alpha$.

Let $V$ be the
$2$-dimensional ${\CC}$-vector
space generated by elements $x_1,x_2$.
We view $V$ as the standard representation of ${\rm GL}_2$.  
We consider elements $f_1 \in V$ and $f_4 \in
{\rm Sym}^4(V)$. If the discriminant of $f_4f_1$ does not vanish,
the equation $y^3f_1=f_4$ defines an equation of a 
smooth projective curve $C$ of genus $3$ with an automorphism
$\alpha$ given by $y\mapsto \rho y$. 
The space $H^0(C,\Omega_C^1)$ comes with a basis
consisting of the forms (in affine coordinates)
$\eta=dx/f_1y$, $\omega_1= dx/y^2$, $\omega_2= dx/f_1y^2$.

An element $(a,b;c,d) \in 
{\rm GL}_2$ acts on $f_4$ and $f_1$ by
$$
f_4(x_1,x_2) \mapsto f_4(ax_1+bx_2,cx_1+dx_2), \quad
f_1(x_1,x_2)\mapsto f_1(ax_1+bx_2,cx_1+dx_2),
$$
and we can define an action
$$
y \mapsto y/(cx_1+dx_2)\, .
$$

However, in order to get the right stack quotient 
we need to consider a twisted action.
We define $V_{m,n}$ for $m\in {\ZZ}_{\geq 0}$ and $n\in {\ZZ}$ 
as the ${\rm GL}_2$-representation
$$
{\rm Sym}^m(V) \otimes \det(V)^{\otimes n} \, .
$$
The underlying space of $V_{m,n}$ can and will be identified with ${\rm Sym}^m(V)$,
but the action of ${\rm GL}_2$ is different.

We define an action of 
${\GG}_m$ on $V_{4,-2} \oplus V_{1,1}$  
by letting $t \in {\GG}_m$ act via $(x_1,x_2)\mapsto (tx_1,tx_2)$. 
Via $y\mapsto t\, y$ this leaves the equation 
(5) unchanged. It corresponds to the action of the diagonal ${\GG}_m$
in $\mathcal{G}$. Then the  action of the diagonal ${\GG}_m$ in 
${\rm GL}_2$ on $V_{4,-2} \oplus V_{1,1}$ 
is given by $(f_4,f_1) \mapsto (f_4,t^3 f_1)$; hence the central
$\mu_3 \subset {\rm GL}_2$ acts trivially.

We let $\mathcal{Y}$ be the subset of $V_{4,-2}\times V_{1,1}$
of pairs $(f_4,f_1)$ such that the discriminant of $f_4f_1$ is not
zero. Moreover, we let ${\PP}(\mathcal{Y})$ be the image in
${\PP}(V_{4,-2}\oplus V_{1,1})$. The stack quotient that we need
is obtained by first dividing by the diagonal ${\GG}_m$ in $\mathcal{G}$
 to get ${\PP}(\mathcal{Y})$ and then dividing by
 the action of $P\mathcal{G} \subset {\rm PGL}(3)$ on
${\PP}(\mathcal{Y})$. Equivalently, we directly take the stack quotient  
$[\mathcal{Y}/\mathcal{G}]$. Summarizing we get the following result.

\begin{proposition}\label{stack-quotient}
The stack quotient $[\mathcal{Y}/\mathcal{G}]$ represents
the moduli stack of curves $\mathcal{N}$.
\end{proposition}
Note that the central $\mu_6 \, {\rm Id}_V$ acts trivially on the
equation (5) but $-1 \, {\rm Id}_V$ acts by $f_1 \mapsto -f_1$.
Hence the stabilizer of a generic element is $\mu_3$ as it should.

\bigskip

We can extend $\mathcal{Y}$ to the open subset $\mathcal{Y}'$ of $V_{4,-2}\oplus
V_{1,1}$ consisting of pairs $(f_4,f_1)$ such that either
\begin{enumerate}
\item{}  $f_4$ 
has one double zero: $f_4=h_1^2h_2$ with $\deg(h_i)=i$ 
and ${\rm disc}(h_2h_1f_1)\neq 0$, or

\item{} 
$f_4=f_1h_3$ with $h_3$ of degree $3$ and  ${\rm disc}(f_4)\neq 0$.
\end{enumerate}
The locus $\mathcal{Y}'$ has a complement of 
codimension $2$ in $V_{4,-2}\oplus V_{1,1}$.

\smallskip

In Case (1) the equation $y^3f_1=h_1^2h_2$ (or equivalently
$y^3=h_1^2h_2 f_1^2$) defines a curve
of genus $2$ which is a triple cyclic cover of ${\PP}^1$.
In Case (2) the equation $y^3=h_3$
defines a $3$-pointed genus $1$ curve $C_1$  with as marked points
the three points of 
the fibre of $C_1\to {\PP}^1$  defined by $f_1=0$. 

The space $\mathcal{N}$ can be viewed as a Hurwitz space and 
can be compactified as such. We will deal with this 
in the next section. 

\smallskip

\begin{remark}\label{ternary} \textbf{Relation with ternary quartics}. 
We conclude this section by giving the relationship with
a stack quotient description of the moduli of non-hyperelliptic
curves of genus $3$.
It is well-known that the  moduli space $\mathcal{M}_{3}^{\rm nh}$ 
of non-hyperelliptic curves of genus $3$ can be described
as a stack quotient associated to the action of 
${\rm GL}_3$ on ternary quartics. 
Since we are using canonical curves as in (5) we get an
embedding of stacks $\mathcal{N} \to \mathcal{M}_3^{\rm nh}$.

Let $W$ be a $3$-dimensional vector space and let
$$
W_{4,0,-1}={\rm Sym}^4(W) \otimes \det(W)^{-1} \, .
$$
This space can be regarded as the space of ternary quartics with a twisted
${\rm GL}(W)$ action. The element $t \, {\rm id}_W$ 
in the diagonal ${\GG}_m$ in
${\rm GL}(W)$ acts via $f \mapsto t\, f$ for a quartic $f \in W_{4,0,-1}$.
We let $\mathcal{Z} \subset W_{4,0,-1}$ be the subset  of quartics with 
non-zero discriminants. The stack quotient $[\mathcal{Z}/{\rm GL}(W)]$
can be identified with $\mathcal{M}_3^{\rm nh}$, see \cite{CFG3}. 

To connect it to our case 
we write $W=W'\oplus W^{\prime\prime}$ with $W'=\langle x_1,x_2\rangle$ and
$W^{\prime\prime}=\langle y \rangle$. 
The element ${\rm diag}(1,1,\rho)$ acts $W_{4,0,-1}$ 
and the $\rho^2$-eigenspace is
$$
{\rm Sym}^4(W')\otimes \det(W')^{-1}\otimes (W^{\prime\prime})^{-1} \bigoplus
W'\otimes \det(W')^{-1} \otimes (W^{\prime\prime})^{\otimes 2} 
$$
By putting $\det(W')=W^{\prime\prime}$ we find the ${\rm GL}(W')$ representation
$(W')_{4,-2}\oplus (W')_{1,1}$. Note that $\det(W')$ and $W^{\prime\prime}$ differ
by the action of the diagonal $\mu_3$ in ${\rm GL}_2$ when viewed as subgroup  of
$\mathcal{G}$.
\end{remark}
\end{section}
\begin{section}{The Hurwitz space}\label{Hurwitz}
We briefly discuss a compactification of $\mathcal{N}$ as a 
Hurwitz space. We consider admissible triple cyclic covers
$f: C \to P$ where $C$ is a nodal curve of genus $3$ and
$P$ a stable curve of genus $0$ with marked points $p_0,
\{p_1,\ldots,p_4\}$ together with an order $3$ automorphism
$\alpha$ such that $C/\alpha$ is isomorphic to $P$,
$f$ equals the map $C\to C/\alpha$ and 
$\alpha$ fixes the $p_i$ while acting by $\rho^2$ (resp.\ by
$\rho$) on the tangent space of $p_0$ (resp.\ of $p_i$ with 
$i\neq 0$).
Here admissible is taken in the sense of Harris-Mumford, see for example
\cite[p. 175]{H-M}.
The marking of $p_1,\ldots,p_4$ is taken unordered, that is,
modulo the action of the symmetric group $\mathfrak{S}_4$.
We denote this space by $\overline{\mathcal{N}}$. 

It allows a morphism $\overline{\mathcal{N}} \to 
\overline{\mathcal{M}}_{0,1+4}
=\overline{\mathcal{M}}_{0,5}/\mathfrak{S}_4$ with 
$\overline{\mathcal{M}}_{g,n}$ the usual Deligne-Mumford moduli stack
of stable $n$-pointed curves of genus $g$.

Note that the moduli space $\overline{\mathcal M}_{0,1+4}$ 
of marked stable curves of genus $0$
has a stratification with five strata
according to the topological type of the genus $0$ curve. 

There is a corresponding stratification of $\overline{\mathcal{N}}$.
We now describe the
five types of curves $(C,P,\alpha)$ corresponding to the strata of
${\overline{\mathcal{M}}_{0,1+4}}$.
These are:
\begin{enumerate}
\item{} $C$ is smooth.
\item{} $C$ is a union $C_1 \cup C_2$ of curves $C_i$ of genus $i$ 
with automorphisms $\alpha_1$ and $\alpha_2$ of order~$3$.
The unique node is a fixed point of $\alpha_1$ and $\alpha_2$.
Moreover, this point is of type $\rho^2$ for
$\alpha_2$ and of type $\rho$ for $\alpha_1$.
\item{} $C$ is a linear chain of three  curves $C_i$ of genus $1$ 
with automorphisms $\alpha_i$ ($i=1,2,3$) 
and the two nodes are fixed points.
Moreover the action of $\alpha_1$ and $\alpha_3$ is by
$\rho$, while for the middle one the action  
by $\alpha_2$ is by $\rho^2$.
\item{} $C$ is a join of a genus $1$ curve $C_1$ 
with an automorphism $\alpha_1$
that acts by $\rho$ and a rational curve $C_0={\PP}^1$ with an 
automorphism $\alpha_0$
that acts by $x \mapsto \rho x$. The curve $C$ is obtained by identifying the
three points of an $\alpha_1$-orbit of length $3$ with 
$1, \rho, \rho^2$ on ${\PP}^1$.
\item{} $C$ consists of the union of a genus $1$ curve $C_1$ with an
order $3$ automorphism $\alpha_1$
and two ${\PP}^1$'s with automorphism $x \mapsto 1/(1-x)$, say $C_0$ and
$C_0'$, that intersect each other in $0$, $1$ and
$\infty$ such that $C_1$ and $C_0'$ are disjoint,
while $C_1$ and $C_0$ intersect in a fixed point of $\alpha_1$. Moreover,
the action of $\alpha_1$ is by $\rho$.
\end{enumerate}

The corresponding strata are denoted by $\mathcal{N}_i$ for $i=1,\ldots,5$ 
with $\mathcal{N}_1=\mathcal{N}$.

The first three cases  represent curves whose generalized Jacobian
is an abelian variety. The dimensions of the strata are
$2$, $1$, $0$, $1$, $0$ respectively.

The strata $\mathcal{N}_2$ and $\mathcal{N}_4$ correspond to the two cases
(1) and (2) of the preceding section.
To connect it with the quartics discussed there, 
one considers for the case 
$\mathcal{N}_2$
the space $H^0(C_2,\Omega^1_{C_2}(2P))$ with $P$ the point of $C_2$
shared with $C_1$. This space has dimension~$3$ and the action of $\alpha_2$ 
on it has
eigenvalues $\rho,\rho,\rho^2$. A choice of basis 
can be used to generate an equation
of type $y^3f_1=h_1^2h_2$ as in the smooth case.
For $\mathcal{N}_4$ one considers $H^0(C_1,O(Q))$ with $Q$ the degree $3$
divisor of intersection points.

On the Hurwitz space $\overline{\mathcal{N}}$
we have the Hodge bundle ${\EE}$.
It allows a decomposition in $\rho$ and $\rho^2$-eigenspaces
of dimension $2$ and $1$.

\end{section}
\begin{section}{The Torelli morphism}
The homology $H_1(C,{\ZZ})$ of a smooth curve $C$ given by
an equation (4), or in other words of type (1) 
of the preceding section, is
a projective ${\ZZ}[\rho]$-module of rank $3$, hence isomorphic to a
direct sum of ideals of ${\ZZ}[\rho]$, and since $F={\QQ}(\rho)$
has class number $1$, it can be identified with $\Lambda=
O_F^3$ with $O_F={\ZZ}[\rho]$.
Moreover, if have chosen an
embedding $F\hookrightarrow {\CC}$ we obtain a $3$-dimensional
complex vector space $W=H_1(C,{\RR})=H_1(C,{\ZZ})\otimes_{\QQ}{\RR}$.
The Jacobian variety of such a curve $C$ is a principally polarized
abelian threefold $W/\Lambda$ with complex multiplication by the
ring of integers $O_F$.
The polarization defines an alternating integral form on the lattice
$\Lambda$. The corresponding  Hermitian form on $W$
may be normalized to the form
$$
z_1z_2'+z_1'z_2+z_3z_3',
$$
where $x\mapsto x'$ corresponds to the Galois automorphism of
$F/{\QQ}$. This form has signature $(2,1)$.

The Torelli map that associates to a curve $C$ its Jacobian
defines a map 
$$\tau:\mathcal{N} \to \Gamma\backslash \mathfrak{B}\, .
$$
Note that our generic curve, given by an equation $y^3f_1=f_4$,
has an automorphism group
of order $3$, while the generic Jacobian of such a curve has
an automorphism group of order $6$. 

Recall that on  $\Gamma\backslash \mathfrak{B}$ we have the 
basic orbifold vector bundle $U$ corresponding to the factor
of automorphy $j_2$. We may identify $\mathcal{N}$ with the
stack quotient $[\mathcal{Y}/\mathcal{G}]$ by Proposition \ref{stack-quotient}
and consider the pullback of $U$ on $\mathcal{Y}\subset V_{4,-2}\times V_{1,1}$.

\begin{proposition}\label{torelli}
The Torelli map $\tau$ induces an orbifold  morphism of degree $2$ with
 the property that the pullback of the orbifold bundle $U$
is the equivariant bundle $V$. 
The image of $\tau$ is the open part where the cusp form $\zeta$ does not
vanish.
\end{proposition}
\begin{proof}
The first statement follows from the construction in Section \ref{StackQuotient}.
The second statement follows from \cite[Thm.\ 6.1.3]{H} or even already 
from Picard's papers \cite{P1,P2,P3}.
\end{proof}

\end{section}

\begin{section}{Teichm\"uller modular forms}\label{TMF}
We begin by noting that we have two notions of modular forms here,
Picard modular forms and Teichm\"uller modular forms. 
On the Hurwitz space $\overline{\mathcal{N}}$  we have 
the Hodge bundle that agrees on 
$$
\mathcal{N}^{\rm ct}=
\mathcal{N}_1\cup \mathcal{N}_2\cup \mathcal{N}_3
$$
(where `ct' refers to compact type) 
with the pullback $\tau^*({\EE})$ under the Torelli morphism
and thus admits a decomposition
$\tau^*({\EE})= \tau^*(U)\oplus \tau^*(L)$. Furthermore,
${\det}(\tau^*(U))$ and $\tau^*(L)$ differ by a torsion line bundle 
$\tau^*(R)^{-1}$.

We will denote $\tau^*({\EE})$ again by ${\EE}$.
We thus can speak of Teichm\"uller modular forms with a character: 
a Teichm\"uller modular form of weight $(j,k,l)$ is a section 
on $\mathcal{N}^{\rm ct}$ of
$$
{\EE}_{j,k,l}={\rm Sym}^j(\tau^*(U)) 
\otimes \tau^*(L)^k \otimes \tau^*(R)^{l}\, .
$$

\begin{proposition}\label{Koecher-type}
A section of ${\EE}_{j,k,l}$ over 
$\mathcal{N}^{\rm ct}$ 
extends to a section of ${\EE}_{j,k,l}$ over
$\overline{\mathcal{N}}$. 
\end{proposition}
The proof is a slight adaptation of the proof of Proposition 14.1
in \cite{CFG3} and is omitted.

\smallskip

We can pull back Picard modular forms via the Torelli map.
Since the Torelli map is of degree $2$, there can be more Teichm\"uller
modular forms than Picard modular forms, that is,  Teichm\"uller modular
forms that are not pullbacks of Picard modular forms.

An example of a Teichm\"uller modular form that is  not
a Picard modular form on $\Gamma$ is the 
form $\zeta^3 \in S_{18}(\Gamma[\sqrt{-3}])$. 
Since $-1_3$ acts trivially on $\mathfrak{B}$ 
and $\zeta$ changes sign under $-1_3$, 
the form $\zeta^3$ does not live on $\Gamma$.
But it lives on $\mathcal{N}$ as we now show.
 
\begin{lemma}
The form $\zeta^3$ is a Teichm\"uller modular form of weight $(0,18,3)$ 
\end{lemma}
\begin{proof} On $\mathcal{M}_3$ we have a Teichm\"uller form of weight $9$,
see \cite{CFG3}. The pullback under the morphism 
$\mathcal{N}^{\rm ct} \to \mathcal{M}_3$
of $\chi_9$ under the morphism $\mathcal{N} \to \mathcal{M}_3$
gives a Teichm\"uller form $\zeta'$ of weight $(0,18,3)$ on 
$\mathcal{N}^{\rm ct}$ 
that does not vanish outside the divisor of $\zeta$. 
As a calculation shows, the pullback of $\chi_9$
via $\iota: \mathfrak{B} \to \mathfrak{H}_3$ 
to $\Gamma[\sqrt{-3}]\backslash \mathfrak{B}$ gives a
non-zero multiple of the modular form 
$\zeta^3 \in S_{18}(\Gamma[\sqrt{-3}])$. Hence $\zeta'$
coincides with a non-zero multiple of $\zeta^3$. 
\end{proof}

\begin{remark}
The form $\zeta^3$ can be constructed algebraically as Ichikawa
 does in \cite[p. 1059]{Ichikawa} for $\chi_9$. 
One observes that the natural map of rank $6$ sheaves
${\rm Sym}^2({\EE})\to 
\pi_{*}(\omega^{\otimes 2}_{\mathcal{C}/\mathcal{N}^{\rm ct}})
$
with $\pi: \mathcal{C} \to \mathcal{N}^{\rm ct}$ the
universal curve,
is an isomorphism on $\mathcal{N}$, and by \cite[Thm 5.10]{Mumford1977} 
taking the determinant thus gives a morphism $L^8\otimes R^{-4} \to L^{26}\otimes
R^{-13}$. This morphism extends over $\mathcal{N}^{\rm ct}$, but
vanishes on $\mathcal{N}_2$. 
This gives a section of $L^{18}\otimes R^3$.
\end{remark}

\begin{remark}\label{level}
If we consider the moduli stack $\mathcal{N}[\Gamma_1]$ of curves of genus $3$
that are a triple cyclic cover of ${\PP}^1$ with a Jacobian with a $\Gamma_1$-level structure,
then $\zeta$ is a Teichm\"uller modular form on $\mathcal{N}[\Gamma_1]$ 
of weight $(0,6,1)$, but it is
not a Picard modular form on $\Gamma_1$. Its square $\zeta^2$ is a Picard modular
form on $\Gamma_1$.
\end{remark}

The involution $-1$ (fibrewise) on ${\EE}$ induces an involution $\theta$
on the space $H^0(\mathcal{N}^{\rm ct},\EE_{j,k,l})$
and on $H^0(\overline{\mathcal{N}}, \EE_{j,k,l})$.
The pullback of $M_{j,k,l}(\Gamma)$ to $\mathcal{N}^{\rm ct}$
lands in the $(+1)$-eigenspace $H^0(\overline{\mathcal{N}},\EE_{j,k,l})^{+}$
of $\theta$. 

\begin{lemma}
We have $H^0(\overline{\mathcal{N}},\EE_{j,k,l})^{+}=
\tau^*(M_{j,k,l}(\Gamma))$.
\end{lemma}
\begin{proof}
We have $H^0(\mathcal{N}^{\rm ct},\tau^*({\EE}_{j,k,l}))^{+}
= \tau^*(H^0(\tau(\mathcal{N}^{\rm ct}),{\EE}_{j,k,l}))
= \tau^*(M_{j,k,l}(\Gamma))$, where the last equality follows from 
the Koecher principle for Picard modular forms. Proposition
\ref{Koecher-type} concludes the proof.
\end{proof}
We now show that Teichm\"uller modular forms can be viewed as Picard modular forms
on a congruence subgroup.
\begin{corollary}\label{Teichmueller=Picard}
For the $(-1)$-eigenspace of $\theta$ we have
$$
H^0(\overline{\mathcal{N}},{\EE}_{j,k,l})^{-}
\cong M_{j,k,l}(\Gamma[\sqrt{-3}])^{s[1^4]}\, .
$$
\end{corollary}
\begin{proof}
Multiplication by $\zeta\in S_{0,6,1}(\Gamma[\sqrt{-3}])$ 
maps $\theta$-anti-invariant forms 
to $\theta$-invariant ones, that is,  Picard modular forms. The fact that
$\zeta$ is $\mathfrak{S}_4$ anti-invariant
completes the proof.
\end{proof}
Recall the character $\epsilon$ obtained from the sign character on
$\mathfrak{S}_4$ as defined in Section \ref{sectionPMF}. We have
$$
M_{j,k}(\Gamma_1,\epsilon)=M_{j,k}(\Gamma_1[\sqrt{-3}])^{s[1^4]}\, .
$$
We define an index $2$ subgroup $\tilde{\Gamma}_1$ 
of $\Gamma_1$ as the kernel of $\epsilon$.
Thus a Teichm\"uller form can be viewed as a Picard modular form on
$\tilde{\Gamma}_1$.

\end{section}
\begin{section}{Covariants of pairs of binary forms of degree $4$ and $1$}\label{covariants}
We recall some classical invariant theory. 
As before we have the ${\CC}$-vector space  $V=\langle x_1,x_2\rangle$ and we write $V_n$ for
${\rm Sym}^n(V)$, the space of binary forms of degree $n$. 
Consider for a given tuple $(n_1,\ldots,n_r)$ the ${\rm GL}_2$-representation 
$$
\mathcal{V}=V_{n_1}\oplus \cdots \oplus V_{n_r}\, .
$$ 
A covariant of $\mathcal{V}$ of order $m$ and degree $d$ is an equivariant polynomial map $\phi: \mathcal{V} \to V_m$ that is homogeneous of degree
$d$: 
$$
\phi(g\cdot v)=g \cdot \phi(v), \quad \phi(tv)=t^d\, \phi(v) \quad
 \text{\rm  for all
$v \in \mathcal{V},\,  t\in {\GG}_m$}.
$$ 
The covariants form a doubly graded ring 
$$
\mathcal{C}(\mathcal{V})=\oplus_{d,m} \mathcal{C}(\mathcal{V})_{d,m}\, .
$$

A covariant of order $m=0$ is called an invariant. One can view it as
a polynomial in the coefficients of the $r$-tuple $(f_1,\ldots,f_r)
\in \mathcal{V}$ of binary forms that is invariant under the action
of ${\rm SL}_2$. 

Classical invariant theory provides 
a ${\rm SL}_2$-equivariant linear map 
from $V_m\otimes V_n \to V_{m+n-2k}$ via $f \otimes g \mapsto (f,g)_k$,
where the expression $(f,g)_k$ is
called the $k$th-transvectant and  
it is given by
\[
(f,g)_k=\frac{(m-k)! \, (n-k)!}{m! \, n!}\sum_{j=0}^k (-1)^j
\binom{k}{j}
\frac{\partial ^k f}{\partial x_1^{k-j}\partial x_2^j}
\frac{\partial ^k g}{\partial x_1^{j}\partial x_2^{k-j}} \, .
\]
Covariants of $\mathcal{V}$ can be identified with the invariants
of $\mathcal{V}\oplus V_1\cong \mathcal{V}\oplus V_1^{\vee}$
via the map that associates to a covariant of order $m$ the
transvectant $(\phi(v),l^m)_m$ with $l \in V_1$.
For a good reference we refer to Draisma \cite{Draisma}.

We are interested in the action of ${\rm GL}_2$ on $V_4\oplus V_1$
and the corresponding covariants. Equivalently, we can look at the
invariants of the action on $V_4\oplus 2\, V_1$.
We write $f$ for the covariant which is the universal binary quartic 
(corresponding to the identity map on $V_4$)
and $h$ and $l$ for the universal linear terms:
 \[
f=a_0x_1^4+a_1x_1^3x_2+a_2x_1^2x_2^2+a_3x_1x_2^3+a_4x_2^4,
\quad
h=b_0x_1+b_1x_2,
\quad
\text{and}
\quad
l=l_0x_1+l_1x_2,
\]
The following result is classical; we refer to \cite{Draisma}.

\begin{proposition}
The $20$ generating invariants of $V_4\oplus 2\, V_1$ are given by:
\begin{align*}
I_{2,1}&=(f,f)_4,
\quad
I_{2,2}=(h,l)_1,\quad
I_{3,1}=(f,(f,f)_2)_4,
\quad
I_{5,1}=(f,h^4)_4,
\quad
I_{5,2}=(f,h^3l)_4,\\
I_{5,3}&=(f,h^2l^2)_4,
\quad
I_{5,4}=(f,hl^3)_4,
\quad
I_{5,5}=(f,l^4)_4,
\quad
I_{6,1}=((f,f)_2,h^4)_4,
\\
I_{6,2}&=((f,f)_2,h^3l)_4,
\quad
I_{6,3}=((f,f)_2,h^2l^2)_4,
\quad
I_{6,4}=((f,f)_2,hl^3)_4,
\quad
I_{6,5}=((f,f)_2,l^4)_4,\\
I_{9,1}&=((f,(f,f)_2)_1,h^6)_6,
\quad
I_{9,2}=((f,(f,f)_2)_1,h^5l)_6,
\quad
I_{9,3}=((f,(f,f)_2)_1,h^4l^2)_6,\\
I_{9,4}&=((f,(f,f)_2)_1,h^3l^3)_6,
\quad
I_{9,5}=((f,(f,f)_2)_1,h^2l^4)_6,
\quad
I_{9,6}=((f,(f,f)_2)_1,hl^5)_6,\\
I_{9,7}&=((f,(f,f)_2)_1,l^6)_6.
\end{align*} 
\end{proposition}

We get the generating covariants of $V_4\oplus V_1$ by substituting $l_0=-x_2$ and
$l_1=x_1$ in the generating invariants of $V_4\oplus 2\, V_1$; 
we denote these covariants by
$J_{a,b,c}$, where $a$ is degree in the coefficients of $f$, $b$ 
is the degree in the coefficients of $h$
and $c$ is the degree in $x_1$ and $x_2$. If we write 
$$
s: \mathcal{C}(V_4\oplus 2\, V_1) \to \mathcal{C}(V_4\oplus V_1)
$$
for this substitution, 
we find the following table for the images under $s$
of the twenty generators:

\begin{footnotesize}
\smallskip
\vbox{
\bigskip\centerline{\def\quad{\hskip 0.6em\relax}
\def\quod{\hskip 0.5em\relax }
\vbox{\offinterlineskip
\hrule
\halign{&\vrule#&\strut\quod\hfil#\quad\cr
height2pt&\omit&&\omit&&\omit&&\omit&&\omit&&\omit&&
\omit&&\omit&&\omit&&\omit&\cr
&$I_{2,1}$ && $I_{2,2}$ && $I_{3,1}$ && $I_{5,1}$ && $I_{5,2}$ && $I_{5,3}$ && $I_{5,4}$ && $I_{5,5}$&& $I_{6,1}$ && $I_{6,2}$ & \cr
\noalign{\hrule}
&$J_{2,0,0}$ && $J_{0,1,1}$ && $J_{3,0,0}$ && $J_{1,4,0}$ && $J_{1,3,1}$ && $J_{1,2,2}$ && $J_{1,1,3}$ && $J_{1,0,4}$&& $J_{2,4,0}$ && $J_{2,3,1}$ & \cr
\noalign{\hrule}
height2pt&\omit&&\omit&&\omit&&\omit&&\omit&&\omit&&
\omit&&\omit&&\omit&&\omit&\cr
\noalign{\hrule}
&$I_{6,3}$ && $I_{6,4}$ && $I_{6,5}$ && $I_{9,1}$ && $I_{9,2}$ && $I_{9,3}$ && $I_{9,4}$ && $I_{9,5}$&& $I_{9,6}$ && $I_{9,7}$ & \cr
\noalign{\hrule}
&$J_{2,2,2}$ && $J_{2,1,3}$ && $J_{2,0,4}$ && $J_{3,6,0}$ && $J_{3,5,1}$ && $J_{3,4,2}$ && $J_{3,3,3}$ && $J_{3,2,4}$&& $J_{3,1,5}$ && $J_{3,0,6}$ & \cr
} \hrule}
}}
\end{footnotesize}

\smallskip

Some simple examples are 
$$
J_{2,0,0}=(12 \, a_0a_4-3\, a_1a_3+a_2^2)/6\, ,
$$
and
$J_{0,1,1}=h$, $J_{1,0,4}=f/70$. The discriminants of $f$ and $fh$ 
are given by 
$$
32 \, (J_{2,0,0}^3-6 J_{3,0,0}^2)\quad {\rm and} \quad
32 \, (J_{2,0,0}^3-6J_{3,0,0}^2)J_{1,4,0}^2 \, .
$$
These invariants satisfy many relations, for example we have
$$
5250\, J_{2,4,0}^3+
26136\, J_{3,6,0}^2+
1750\, J_{1,4,0}^3 J_{3,0,0}-
2625\,  J_{1,4,0}^2 J_{2,0,0} J_{2,4,0}=0\, . \eqno(6)
$$

\begin{remark}\label{concomitants}
We may view $V$ as the dual of $V^{\vee}$; then we can view
${\rm Sym}^m(V)$ as the set of homogeneous polynomial maps 
$V^{\vee}\to {\CC}$ of degree $m$, and as such it carries a
natural left action of ${\rm GL}(V)$ by composition. 
Instead of the representation $\mathcal{V}=\oplus V_{n_i}$ 
we can also consider  twisted cases $\mathcal{V}=\oplus_{i=1}^r V_{n_i,m_i}$
with, as before, 
$V_{n,m}= {\rm Sym}^n(V) \otimes \det(V)^{\otimes m}$
and consider covariants for this ${\rm GL}_2$-representation.
That is, taking $\mathcal{V}^{\vee}= \oplus V^{\vee}_{n_i,m_i}$  
we look at ${\rm GL}_2$-equivariant embeddings
$$
V_{j,k} \hookrightarrow  \mathcal{O}(\mathcal{V}^{\vee})=
 \oplus_m {\rm Sym}^m(\mathcal{V})\, .
$$
\end{remark}
\end{section}

\begin{section}{From covariants to modular forms}

A Teichm\"uller modular form of weight 
$(j,k,l)$ is a section of 
${\rm Sym}^j(U) \otimes \det(U)^k \otimes R^{l+k}$ on 
$\mathcal{N}^{\rm ct}$.  Here we write again $U$ for
$\tau^*(U)$.
We identify the Hurwitz space $\mathcal{N}$ with the
quotient stack $[\mathcal{Y}/\mathcal{G}]$.
Under the Torelli morphism $\tau: [\mathcal{Y}/\mathcal{G}]\cong 
\mathcal{N}\to \Gamma\backslash \mathfrak{B}$
the pullback of the bundle $U$ is the equivariant bundle $V$.
Therefore modular forms pull back to bi-covariants for
the action of $\mathcal{G}$ on $\mathcal{Y}$.
If the modular form is of weight $(j,k,l)$, that is, a section
of ${\rm Sym}^j(U)\otimes \det(U)^k \otimes R^{l+k}$,  the
corresponding bi-covariant lies by Remark \ref{concomitants}
in a space of bi-covariants that is given 
as the image of ${\rm GL}_2$-equivariant map
$$
V_{j,k} \to \mathcal{O}(V_{4,-2}\oplus V_{1,1})\, ,
$$
where $\mathcal{O}(V_{4,-2}\oplus V_{1,1})$ is the ring of polynomial
functions on $V_{4,-2}\oplus V_{1,1}$. The character of the modular
form can be read off from the action of the diagonal ${\GG}_m \subset {\rm GL}_2$. 

Thus we see that a section of 
${\rm Sym}^j(U)\otimes L^k \otimes R^{l}$ on 
$\Gamma\backslash \mathfrak{B}$
pulls back to a covariant for 
the action of ${\rm GL}_2$ on $V_{4,-2}\oplus V_{1,1}$, and this 
covariant can be identified with a covariant for the (untwisted) action
on $V_4\oplus V_1$. Moreover,
we are identifying covariants of the action of ${\rm GL}_2$
on $V_{4}\oplus V_{1}$ with invariants of the action on 
$V_{4}\oplus V_{1} \oplus V_1$ as explained in Section \ref{covariants}.
Thus our section provides a covariant  
$J_{a,b,c}$, where the index $(a,b,c)$  indicates that it has 
degree $a$ in the $a_i$, degree $b$ in the $b_i$ and degree $c$ in $x_1,x_2$.
Clearly, we have  $j=c$. Moreover, we find  $k=(3b-c)/2$,
since the action of the diagonal ${\GG}_m \subset {\rm GL}_2$ is
by $t^0$ on $V_{4,-2}$, by $t^3$ on $V_{1,1}$, by $t^{-1}$ on 
the component $V_1$, the dual of $V_{1,1}$ but twisted back by ${\det}^{-1}$, 
and by $t^2$ on ${\det}(V)$.
If we start with a Picard modular form, then $a+b+c$ is even.

Let 
$$
{\MM}=\oplus_{j,k,l} H^0(\overline{\mathcal{N}},{\EE}_{j,k,l})
$$ 
be the ring of modular forms, where the ring structure is obtained in a similar way as 
for Picard modular forms,
see Section \ref{sectionPMF}. Restricting to $\mathcal{N}$ 
we get a map
$$
{\MM} \langepijl{\mu} \mathcal{C}(V_4\oplus V_1) \, .
$$
Since the image of the Torelli map on $\mathcal{N}$ 
is the complement of $T_1$,
the locus where the cusp form $\zeta$ vanishes, 
as discussed in Section \ref{ModularCurves}, 
a covariant defines a meromorphic 
modular form that is holomorphic outside this divisor. 
Thus we can complement the map $\mu$ by a ring homomorphism
$$
{\MM} \langepijl{\mu} \mathcal{C}(V_4\oplus V_1)
\langepijl{\nu} {\MM}[1/\zeta]  \eqno(7)
$$
with the property that $\nu\circ \mu = {\rm id}_{\MM}$. Note that
$\zeta$ is a Teichm\"uller modular form of weight
$6$ with character ${\det}$ as explained in Remark \ref{level}.

On our quotient stack $\mathcal{N}$ 
we have two diagonal sections 
corresponding to the universal quartic
$f_4$ and universal linear form $f_1$.
We put
$$
\chi_{4,-2}=\nu(f_4), \quad \chi_{1,1}=\nu(f_1) \, .
$$
Here $\chi_{4,-2}$ (resp.\ $\chi_{1,1}$) 
is a meromorphic Teichm\"uller modular form of weight $(j,k,l)=(4,-2,1)$
(resp.\ $(1,1,1)$) and we wish to identify these 
meromorphic modular forms.
For this we need an estimate on the pole orders
along the curve $T_1$.  By Corollary \ref{Teichmueller=Picard} 
we may view these
Teichm\"uller forms as Picard modular forms on $\Gamma_1[\sqrt{-3}]$.

\begin{lemma}\label{pole-order}
The meromorphic modular form $\chi_{4,-2}$ has order $-1$
along the curve $T_1$. The meromorphic
modular form $\chi_{1,1}$ is holomorphic.
\end{lemma}
\begin{proof} 
In order to prove that the order of both $\chi_{1,1}$ and $\chi_{4,2}$
is at least $-1$ we may use the restriction of the Teichm\"uller
modular form $\chi_{4,0,-1}$ constructed in  \cite{CFG3}. 
It is known that it has a pole
of order $1$ along the hyperelliptic locus in $\mathcal{M}_3$.
The relation between $\chi_{4,0,-1}$ and the pair 
$(\chi_{4,-2}, \chi_{1,1})$ 
is provided by Remark \ref{ternary}.
From this we can conclude the result since the order of $\chi_{1,1}$
satisfies a congruence condition, see (8) below.
But we shall give a direct argument that gives more information.

We may view $\chi_{1,1}$ and $\chi_{4,-2}$ as 
meromorphic Picard modular
forms on $\Gamma_1[\sqrt{-3}]$. We start with the Taylor expansion
along $T_1$ given by $u=0$
$$
\chi_{1,1}=
\left( \begin{matrix} \chi_{1,1}^{(0)} \\ \chi_{1,1}^{(1)}\\ \end{matrix}
\right) \qquad \text{\rm with} 
\quad \chi_{1,1}^{(m)}=\sum_{n\geq r} f_n^{(m)} u^n \, .
$$
We assume that $\chi_{1,1}$ has order $r$ along $T_1$. Using the
action of $(1,1,\rho)$ and $-1_3$ 
we see that a non-zero term $f_n^{(m)}$  satisfies the congruence condition
$$
n\equiv m \, (\bmod \, 6)\, . \eqno(8)
$$ 
In particular, by a slight variant 
of Proposition
\ref{quasi-modularity}, a non-zero term $f_r^{(m)}$  
is a meromorphic modular form of weight $1+m+r$ on $\Gamma_1(3)$, 
regular outside (the orbit of) $\tau_0=(1-\rho^2)/3$
and with order at least  $r$ at $\tau_0$. 
The space $M_1^{(1)}(\Gamma_1(3))$ is generated
by $\vartheta=\sum_{\alpha \in O_F} q^{N(\alpha)}$ and we know  that
$M_r^{(r)}(\Gamma_1(3))$ is generated by $\vartheta^r$, as
explained in (4) at the end of Section \ref{embedding}.
This implies that $f^{(m)}_r$  is divisible by $\vartheta^r$,
so  $f_r^{(m)}=\vartheta^r \varphi$ 
with $\varphi \in M_{m+1}(\Gamma_1(3))$ for $m=0$ or $1$.
But then $\varphi$ is a non-zero multiple of $\vartheta^{m+1}$, 
implying that 
$f_r^{(m)}$ is a non-zero multiple of $\vartheta^{r+m+1}$.
The anti-invariance of $\chi_{1,1}$ implies that $r\equiv 0 \, (\bmod \, 6)$.
Thus the order of $\chi_{1,1}$ equals the order of $\chi_{1,1}^{(0)}$.

Since holomorphic Picard modular forms have weight 
$(j\geq 0,k \geq 0)$, the order $s$ of   
$\chi_{4,-2}$ along $T_1$ is negative.
If we write $\chi_{4,-2}^{(m)}= \sum_{n\geq s} g_n^{(m)} u^n$
for $m=0,\ldots,4$, then for non-zero $g_n^{(m)}$ we have 
$$ 
n\equiv m+3 \, (\bmod \, 6)\, . \eqno(9) 
$$ 
Moreover, $g_s^{(m)} \in M_{-2+s+m}^{(s)}(\Gamma_1(3))$.
For non-zero $M_{k}^{(s)}(\Gamma_1(3))$ 
we need $-2+s+m\geq s$, hence using the congruence restriction (9), we see
$m=2$ for non-zero $g_s^{(m)}$ and it is a non-zero multiple of
$\vartheta^s$. 
We then have that
$g_{s+1}^{(m)}=0$ unless $m=3$, and in fact it is quasi-modular
and one observes that it is not zero by 
applying the Equation \ref{vergelijkingen} to 
$\vartheta^{s}$. 
Similarly $g_{s+2}^{(m)}=0$ unless $m=4$. Using again 
Equation \ref{vergelijkingen} we see 
${\rm ord}(\chi_{4,-2}^{(4)})= {\rm ord}(\chi_{4,-2}^{(2)})+2$.

The discriminant of $f_4$ and of $f_4f_1$ are invariants that define
scalar-valued modular forms. These invariants are given by 
$\Delta_4=J_{2,0,0}^3-6 J_{3,0,0}^2$
and  $\Delta_4 \, J_{1,4,0}^2$ up to non-zero multiplicative scalars.
The weight of $\nu(\Delta_4)$ is $0$ and that of $\nu(J_{1,4,0})$
is $6$, and these are units outside $T_1$. Therefore $\nu(\Delta_4)$
is constant and $\nu(J_{1,4,0})$ must be a multiple of $\zeta$.
Now we have 
$J_{1,4,0}= a_0b_1^4-a_1b_0b_1^3+a_2b_0^2b_1^2-a_3b_0^3b_1+a_4b_0^4$
and we can vary $a_4$ and $b_0$, while keeping $b_1$ and $a_0,\ldots,a_3$
fixed. Then the term $a_4b_0^4$ must yield under $\nu$ a regular expression 
and we infer that 
$$
{\rm ord}(g^{(4)})+ 4 \, {\rm ord}(f^{(0)})= s+2+4\, r \geq 0\, ,
$$
where we write $g^{(4)}=\chi_{4,-2}^{(4)}$ and $f^{(0)}=\chi_{1,1}^{(0)}$,
hence $r\geq 0$ and $\chi_{1,1}$ is holomorphic.
Since $\dim M_{1,1}(\Gamma[\sqrt{-3}],{\det})=1$ 
we can identify $\chi_{1,1}$ with a non-zero multiple 
of a  generator $E_{1,1}$ of $M_{1,1}(\Gamma[\sqrt{-3}],{\det})$
constructed in \cite{C-vdG}. 
We conclude from the development given there 
that it has order $0$ with $f_0^{(0)}$
being a non-zero multiple of $\vartheta$ and that $\chi_{1,1}^{(1)}$
has order~$1$.

Keeping now $b_0,b_1$ fixed and varying one of $a_0,\ldots,a_4$, we see
that all terms $a_0b_1^4$, $a_1b_0b_1^3$, $a_2b_0^2b_1^2$, $a_3b_0^3b_1$
in $J_{1,4,0}$ must give regular forms.
Using this regularity we find that
$$
{\rm ord}(g^{(0)},\ldots,g^{(4)})= (\geq 1,\geq 1, \geq -1, \geq 0,\geq 1)
\, .
$$

But using the congruence condition (9) we see
$$
{\rm ord}(g^{(0)},\ldots,g^{(4)})= (\geq 3,\geq 4, = -1, \geq 0,\geq 1)\, ,
$$
which proves that the order of $\chi_{4,-2}$ along $T_1$ equals $-1$.
\end{proof}

\begin{corollary}
The modular form $\chi_{1,1}$ generates $M_{1,1}(\Gamma[\sqrt{-3}],{\det})$.
The modular form $\zeta \chi_{4,-2}$ generates $M_{4,4}(\Gamma[\sqrt{-3}],{\det}^2)$.
\end{corollary}
\begin{proof}
By \cite{C-vdG} we know that  $\dim M_{1,1}(\Gamma[\sqrt{-3}],{\det})=1$.
In \cite{B-vdG} it was shown that 
$\dim S_{4,4}(\Gamma[\sqrt{-3}],\det{}^2)=1$.
As $\zeta$ vanishes along the curve $T_1$ 
it follows that $\zeta \chi_{4,-2}$ is regular and
generates $S_{4,4}(\Gamma[\sqrt{-3}],{\det}^2)$.
\end {proof}

A generator $\chi_{4,4}$  of $S_{4,4}(\Gamma[\sqrt{-3}],{\det}^2)$ 
will be constructed explicitly in Section 
\ref{Construction}. 
In the paper \cite{C-vdG} we constructed 
explicitly an Eisenstein
series $E_{1,1} \in M_{1,1}(\Gamma[\sqrt{-3}],\det{})$. Hence up
to a non-zero multiplicative constant $\chi_{1,1}$ agrees with $E_{1,1}$.

We can write the meromorphic modular form $\chi_{4,-2}$, 
when viewed as a meromorphic Picard modular form on $\Gamma_1[\sqrt{-3}]$,
 as
$$
\chi_{4,-2}=\sum_{i=0}^4 \alpha_i X_1^{4-i}X_2^i \, , \eqno(10)
$$
where the $X_1, X_2$ are dummy variables to indicate the
coordinates of $V$ and the $\alpha_i$
are meromorphic functions on the the $2$-ball $\mathfrak{B}$.
Similarly, we can write $E_{1,1}$ as
$$
E_{1,1}= \beta_1 X_1+\beta_2X_2 \, \eqno(11)
$$
with $\beta_i$ holomorphic on $\mathfrak{B}$.

In Section \ref{Restriction-section} we shall derive the beginning of 
the Taylor expansion along $u=0$
of the generators 
$\chi_{4,4}$ of $S_{4,4}(\Gamma,{\det}^2)$ and $E_{1,1}$ of 
$M_{1,1}(\Gamma,{\det})$.
This gives the orders, see Corollary \ref{orders2}. 
As a corollary we find the orders of the $\alpha_i$ and $\beta_i$
along the curve given by $u=0$.

\begin{corollary}\label{orders}
The orders of $(\alpha_0,\ldots,\alpha_4)$ and $(\beta_1,\beta_2)$
along $T_1$ are
$$
{\rm ord}(\alpha_0,\alpha_1,\alpha_2,\alpha_3,\alpha_4)=
(3,4,-1,0,1)\, , \quad
{\rm ord}(\beta_1,\beta_2)=(0,1)\, .
$$
\end{corollary}
\begin{proof} As observed in Section \ref{ModularCurves} the modular form $\zeta$
vanishes simply along the components of $T_1$ on $\Gamma_1[\sqrt{-3}]\backslash
\mathfrak{B}$. Together with the orders of $\chi_{4,4}$ this proves the result.
\end{proof}

Now we can describe the map $\nu$. 
Recall that we write a covariant as a polynomial
of degree $a$ in the coefficients $a_i$ of $f_4$, of degree $b$ in the
coefficients of $f_1$ and degree $c$ in
$x_1,x_2$. The map $\nu$ amounts to substituting the coordinates
$\alpha_i$ ($i=0,\ldots,4$) and $\beta_i$ ($i=0,1$) in a covariant.
For simplicity we will view the elements of ${\MM}[1/\zeta]$ as Picard modular forms
on $\Gamma_1[\sqrt{-3}]$, see the description in Section \ref{TMF}.

\begin{theorem}\label{SubsThm} The map 
$\nu: \mathcal{C}(V_4\oplus V_1) \to {\MM}[1/\zeta]$
is given by substituting
$\alpha_i$ for $a_i$, $\beta_i$ for $b_i$ and $X_i$ for $x_i$
in a covariant.
The map $\nu$ sends an invariant $J_{a,b,c}$ of degree $a$ in the $a_i$, 
degree $b$ in the $b_i$ and degree $c$
in $x_1,x_2$ to a meromorphic modular form of weight $(j,k,l)=(c, (3b-c)/2, 2(a+b+c))$.
The form $\nu(f)$ is $\mathfrak{S}_4$-invariant if $a+b$ even, and
$\mathfrak{S}_4$-anti-invariant if $a+b$ odd.
\end{theorem}
\begin{proof} This follows from the identities (10) and (11).
\end{proof}
\begin{corollary}
All modular forms on $\Gamma$ can be constructed by substituting
the coordinates of
$\chi_{4,-2}$ and  $E_{1,1}$ in covariants.
\end{corollary}
\begin{proof}
The composition 
$\MM \to \mathcal{C}(V_4\oplus V_1) \to
\MM[1/\zeta]$ is the identity on ${\MM}$. Indeed, the map $\mu$ interprets
modular forms in terms of bi-covariants and the map $\nu$ re-interprets
a bi-covariant as a (a priori meromorphic) Teichm\"uller modular form. But $\nu$
is given by substituting the coordinates of $\chi_{4,-2}$ and $\chi_{1,1}$.
The form $\chi_{1,1}$ is a non-zero multiple of $E_{1,1}$.
\end{proof}
\begin{remark}
Given generators $\chi_{4,4}$ of $S_{4,4}(\Gamma[\sqrt{-3}],{\det}^2)$ and
$E_{1,1}$ of $M_{1,1}(\Gamma[\sqrt{-3}],{\det})$ one can show directly
using the modular behavior of $\chi_{4,4}, E_{1,1}$ and $\zeta$,
 that if we write
$$
\chi_{4,4}/\zeta=\sum_{i=0}^4 \alpha'_i X_1^{4-i}X_2^i,\quad
\chi_{1,1}=\beta_1' X_1+\beta_2' X_2 \, ,
$$
the substitution of $\alpha_i'$ for $a_i$, $\beta_i'$ for $b_i$ and
$X_i$ for $x_i$ in a covariant of multi-degree $(a,b,c)$ gives 
a modular form of weight $(c,(3b-c)/2, 2(a+b+c))$.
\end{remark}

The orders of the modular forms $\nu(J)$ along the curve $T_1$ 
for the generating invariants 
given in Section \ref{covariants} can be deduced from
Corollary \ref{orders}. We give a table.

\begin{footnotesize}
\smallskip
\vbox{
\bigskip\centerline{\def\quad{\hskip 0.6em\relax}
\def\quod{\hskip 0.5em\relax }
\vbox{\offinterlineskip
\hrule
\halign{&\vrule#&\strut\quod\hfil#\quad\cr
height2pt&\omit&&\omit&&\omit&&\omit&&\omit&&\omit&&
\omit&&\omit&&\omit&&\omit&\cr
\noalign{\hrule}
&$J_{2,0,0}$ && $J_{0,1,1}$ && $J_{3,0,0}$ && $J_{1,4,0}$ && $J_{1,3,1}$ && $J_{1,2,2}$ && $J_{1,1,3}$ && $J_{1,0,4}$&& $J_{2,4,0}$ && $J_{2,3,1}$ & \cr
\noalign{\hrule}
&$-2$ && $0$ && $-3$ && $1$ && $0$ && $-1$ && $-1$ && $-1$&& $0$ && $-1$ & \cr
\noalign{\hrule}
height2pt&\omit&&\omit&&\omit&&\omit&&\omit&&\omit&&
\omit&&\omit&&\omit&&\omit&\cr
\noalign{\hrule}
&$J_{2,2,2}$ && $J_{2,1,3}$ && $J_{2,0,4}$ && $J_{3,6,0}$ && $J_{3,5,1}$ && $J_{3,4,2}$ && $J_{3,3,3}$ && $J_{3,2,4}$&& $J_{3,1,5}$ && $J_{3,0,6}$ & \cr
\noalign{\hrule}
&$-2$ && $-2$ && $-2$ && $0$ && $-1$ && $-1$ && $-1$ && $-1$&& $-1$ && $-1$ & \cr
} \hrule}
}}

\end{footnotesize}

\end{section}
\begin{section}{Gradients of theta functions}
In order to use Theorem \ref{SubsThm} effectively it is important to
know the Fourier-Jacobi expansions of $\zeta$, $\chi_{4,-2}$ (or $\chi_{4,4}$)
and $E_{1,1}$ quite well. In this section and the next one we construct
these modular forms and give part of their Fourier-Jacobi expansion.
We will use gradients of theta series to construct modular forms. 

Recall the definition of
theta series with characteristics (see \cite{Igusa-Th}, p.49):
let $g\in\ZZ_{\geqslant 1}$, $(\mu_1, \ldots, \mu_g)\in \RR^g$, 
$(\nu_1, \ldots, \nu_g)\in \RR^g$ and set for
$\tau\in\mathfrak{H}_g$, the Siegel upper half space of degree $g$,  
and $z=(z_1, \ldots, z_g)\in \CC^g$
\[
\vartheta_{\left[\begin{smallmatrix}\mu \\ \nu \end{smallmatrix}\right]}(\tau,z)=
\vartheta_{\left[\begin{smallmatrix} \mu_1, \ldots, \mu_g\\ \nu_1, \ldots, \nu_g \end{smallmatrix}\right]}(\tau,z)=
\sum_{n=(n_1, \ldots, n_g)\in \ZZ^g}
e^{\pi i(n+\mu)\left(\tau(n+\mu)^t+2(z+\nu)^t\right)} \, .
\]
 We simply call this series a theta series with characteristics. 
Its restriction to $z=0$ is called a theta constant,  
and we omit the variable $z$ in this case:
$
\vartheta_{\left[\begin{smallmatrix}\mu \\ \nu \end{smallmatrix}\right]}(\tau,0)=
\vartheta_{\left[\begin{smallmatrix}\mu \\ \nu \end{smallmatrix}\right]}(\tau)
$.
We have the formulas (see loc.\ cit.)
$$
\vartheta_{-\left[\begin{smallmatrix}\mu \\ \nu \end{smallmatrix}\right]}(\tau,z)=
\vartheta_{\left[\begin{smallmatrix}\mu \\ \nu \end{smallmatrix}\right]}(\tau,-z)
\qquad \text{\rm and} \qquad
\vartheta_{\left[\begin{smallmatrix}\mu +m\\ \nu+n \end{smallmatrix}\right]}(\tau,z)=
e^{2 \pi i \mu n^t}
\vartheta_{\left[\begin{smallmatrix}\mu \\ \nu \end{smallmatrix}\right]}(\tau,z)
$$
for any $(m,n)\in\ZZ^g\times\ZZ^g$.  

From \cite[p.\ 85, p.\ 176]{Igusa-Th} 
we deduce the following transformation formula for the gradient 
$\nabla \vartheta_{\left[\begin{smallmatrix}\mu \\ \nu 
\end{smallmatrix}\right]}$
(written as a column vector) of the theta function 
$\vartheta_{\left[\begin{smallmatrix}\mu \\ \nu \end{smallmatrix}\right]}$
as a function of $z$. 

\begin{proposition}\label{gradientthetatrans}
Let $\gamma=( a,b; c, d) \in {\rm Sp}(2g, {\ZZ})$ then we have
$$
\nabla \vartheta_{\gamma\cdot\left[\begin{smallmatrix}\mu \\ \nu \end{smallmatrix}\right]}
(\gamma\, \tau,z(c\tau+d)^{-1})= j(\gamma,\tau,z) \, \left(
(c\tau+d)\nabla\vartheta_{\left[\begin{smallmatrix}\mu \\ \nu \end{smallmatrix}\right]}(\tau,z)+
\vartheta_{\left[\begin{smallmatrix}\mu \\ \nu \end{smallmatrix}\right]}(\tau,z)
\nabla(e^{ \pi i z(c\tau+d)^{-1}c\, z^t})
\right)\, ,
$$
where
$$
j(\gamma,\tau,z) =
\kappa(\gamma)
e^{\pi i \phi(\gamma,\left[\begin{smallmatrix}\mu \\ \nu \end{smallmatrix}\right])}
\det(c\tau+d)^{\frac{1}{2}}
e^{ \pi i z(c\tau+d)^{-1}c\, z^t}
$$
with
\[
\phi(\gamma,\left[\begin{smallmatrix}\mu \\ \nu \end{smallmatrix}\right])=
-\mu d^tb\mu^t+2\mu b^t c\nu^t-\nu c^ta\nu^t+(\mu d^t-\nu c^t)(ab^t)_0
\]
\[
\gamma\cdot\left[\begin{smallmatrix}\mu \\ \nu \end{smallmatrix}\right]=
\left[
\left(\begin{smallmatrix} d & -c\\ -b & a \end{smallmatrix}\right)
\left(\begin{smallmatrix} \mu \\ \nu \end{smallmatrix}\right)
+\frac{1}{2}\left(\begin{smallmatrix} (cd^t)_0 \\ (ab^t)_0 \end{smallmatrix}\right)
\right]=
\left[
\begin{smallmatrix} 
d\mu-c\nu+(cd^t/2)_0\\ 
-b\mu+a\nu+(ab^t/2)_0
\end{smallmatrix}
\right].
\]
Moreover $\kappa(\gamma)$ is an eighth root of unity (depending only on $\gamma$) and the symbol
$X_0$ denotes the diagonal (column) vector (in its natural order) of a matrix $X$.
\end{proposition}

As a direct corollary of Proposition \ref{gradientthetatrans}, we have
for theta functions vanishing on $\mathfrak{B}\times (0) \subset
\mathfrak{B}\times {\CC}^3$ the following.

\begin{corollary}\label{TransformationGradientRestriction}
Assume that
$\vartheta_{\left[\begin{smallmatrix}\mu \\ \nu \end{smallmatrix}\right]}(\iota(u,v),0)=0$
for any $(u,v)\in \mathfrak{B}$. 
Then for any
$\gamma=\left(\begin{smallmatrix} a & b\\ c & d 
\end{smallmatrix}\right)\in {\rm Sp}(6,{\ZZ})$
we have
\[
\nabla \vartheta_{\gamma\cdot\left[\begin{smallmatrix}\mu \\ \nu \end{smallmatrix}\right]}
(\gamma \, \iota(u,v),0)=
\kappa(\gamma)
e^{\pi i \phi(\gamma,\left[\begin{smallmatrix}\mu \\ \nu \end{smallmatrix}\right])}
\det(c\, \iota(u,v)+d)^{\frac{1}{2}}
(c\, \iota(u,v)+d)\nabla\vartheta_{\left[\begin{smallmatrix}\mu \\ \nu \end{smallmatrix}\right]}(\iota(u,v),0).
\]
\end{corollary}

\medskip

If $C$ is a smooth projective curve of genus $3$ that is a 
triple cyclic cover of $C/\alpha={\PP}^1$,
the kernel of multiplication by $1-\alpha$ on
its Jacobian ${\rm Jac}(C)$ is isomorphic to $({\ZZ}/3{\ZZ})^3$.
This is a totally isotropic subspace for the Weil pairing
on ${\rm Jac}(C)[3]$. If the ramification points of
$C \to {\PP}^1$ are $p_0,p_1,\ldots,p_4$, with $p_0$ the
unique point with $\rho^2$ action, we have a surjective
homomorphism
$$
\phi: ({\ZZ}/3{\ZZ})^4 \to {\rm Jac}(C)[1-\alpha],
\quad (c_1,\ldots,c_4) \to \sum c_i (p_i-p_0)
$$
with kernel given by $\sum_{i=1}^4 c_i =0$.
The group $\mathfrak{S}_4 \times \mu_2$ acts and can be
seen as the orthogonal group for $(O_F/\sqrt{-3}\, O_F)^3$ with
the form $ab'+ab'+cc'$.

We can identify the canonical theta divisor
$\Theta \subset {\rm Pic}^{(2)}(C)$
with a theta divisor in ${\rm Jac}(C)$
by translation over $\kappa=2p_0$, a  half-canonical divisor.
There are fifteen $(1-\alpha)$-torsion points lying on a
theta divisor $\Theta- 2p_0 \subset {\rm Jac}(C)$, namely:
$p_i+p_j-2p_0$ for $0\leq i\leq j\leq 4$ and $(i,j)\neq (0,0)$.
Note the linear equivalence
$p_1+p_2+p_3+p_4 \sim 4p_0$.

The set of $15$ torsion points on the theta divisor
can be divided into three sets:
$$C_1=\{0, \pm (p_i-p_0), i=1,\ldots4\}, \quad
C_2=\{p_i+p_j-2p_0: \, 1\leq i < j\leq 4\}$$
and $C_0$ the complement of $C_1\sqcup C_2$
with cardinalities $\# C_0=12$, $\# C_1=9$
and $\# C_2=6$.

If we let the theta characteristic $2p_0$ correspond to 
$\left[\begin{smallmatrix}\mu \\ \nu \end{smallmatrix}\right]=
\left[\begin{smallmatrix}0 & 1/2 & 0 \\ 0 & 1/2 & 0 
\end{smallmatrix}\right]$,
then using the embedding $\sigma$ 
of $\Gamma[\sqrt{-3}]$ in ${\rm Sp}(6,{\ZZ})$
and the property that 
\[
\sigma(g)\cdot \left[\begin{smallmatrix}\mu \\ \nu \end{smallmatrix}\right]\equiv
\left[\begin{smallmatrix}\mu \\ \nu \end{smallmatrix}\right] \bmod \ZZ
\]
for any  of $g\in \Gamma[\sqrt{-3}]$, we find that the set
$C_i$ of $(1-\alpha)$-torsion points 
corresponds to the set of theta characteristics of degree $3$
$$
\left\{ \left[ \begin{smallmatrix} k/3 & (2l+1)/6 & -k/3 \\
m/3 & (2l+1)/6 & m/3 \end{smallmatrix} \right] : 0\leq k,l,m \leq 2:
km-(l-1)^2\equiv i-1 (\bmod 3\right\} \, .
$$
We will abbreviate such a characteristic by $[k l m]$. 
\begin{lemma}
Writing $\sigma(g)=\left(\begin{smallmatrix} a & b\\ c & d \end{smallmatrix}\right)$
for $g\in \Gamma[\sqrt{-3}]$ and $\iota(u,v)=\tau$ for $(u,v) \in \mathfrak{B}$ we have
\[
(c\tau +d)
\left[
\begin{smallmatrix}
v_1 \\
\left(\begin{smallmatrix} 1 & 0\\ 0 & -\rho^2 \end{smallmatrix}\right)
\left[
\begin{smallmatrix}
v_2 \\
v_1
\end{smallmatrix}
\right]
\end{smallmatrix}
\right]=
\left[
\begin{smallmatrix}
{\rm pr}_2(j_2(g,(u,v))
\left[
\begin{smallmatrix}
v_2 \\
v_1
\end{smallmatrix}
\right])\\
\left(\begin{smallmatrix} 1 & 0\\ 0 & -\rho^2 \end{smallmatrix}\right)
j_2(g,(u,v))
\left[
\begin{smallmatrix}
v_2 \\
v_1
\end{smallmatrix}
\right]
\end{smallmatrix}
\right] \, ,
\]
where ${\rm pr}_2$ denotes
the projection $\left[
\begin{smallmatrix}
x \\
y
\end{smallmatrix}
\right] \mapsto y$.
\end{lemma}
The proof can be carried out by checking it on the generators of $\Gamma[\sqrt{-3}]$.
A set of generators was given in Lemma \ref{generators}.

From the set $C_1$ we take representatives modulo changing sign
$$
C_1'= \{ [011], [110], [101], [202], [010] \}
$$
and for $\lambda \in C_1'$  with $g_0=\rho 1_3$ and $\sigma(g_0)=(a_0,b_0;c_0,d_0) \in {\rm Sp}(6,{\ZZ})$ we have
\[
\nabla \vartheta_{\lambda}(\iota(u,v),0)=
\rho
(c_0\iota(u,v)+d_0)
\nabla \vartheta_{\lambda}(\iota(u,v),0)
=
\left[\begin{smallmatrix}0 & 0 & -\rho\\-u & 1 & -\rho u\\ \rho & 0 &-\rho\end{smallmatrix}\right]
\nabla \vartheta_{\lambda}(\iota(u,v),0)
\]
and this implies that
$$
\frac{\partial \vartheta_{\lambda}}{\partial z_3}(\iota(u,v),0)=
-\rho^2\, \frac{\partial \vartheta_{\lambda}}{\partial z_1}(\iota(u,v),0).
$$
So for the characteristics in the set $C_1$, only the last component of the gradient of $\vartheta_{\lambda}$ 
is dependent of the previous one. This gives a $(2+1)$--decomposition and therefore 
good hopes to get vector-valued modular forms. 

For the five elements $\lambda_i \in C_1'$ ($i=0,\ldots,4$) listed in the order above we put
$$
F_i(u,v)=c^{-1}\,
\left[
\begin{matrix}
\frac{\partial \vartheta_{\lambda_i}}{\partial z_2}(\iota(u,v),0)\\
\frac{\partial \vartheta_{\lambda_i}}{\partial z_1}(\iota(u,v),0)
\end{matrix} 
\right] \eqno(12)
$$
with  the constant $c$ given by
$$
c=\vartheta_{\left[\begin{smallmatrix} 1/6 \\ 1/6 \end{smallmatrix}\right]}(-\rho^2,0)=\frac{3^{3/8}}{2 \pi}\Gamma(1/3)^{3/2}e^{\frac{5\pi i}{72}}\, .
\eqno(13)
$$

Using Corollary \ref{TransformationGradientRestriction} we then obtain:
\begin{lemma}
Let $\lambda=[k l m]\in C_1^{\prime}$ and $g \in \Gamma[\sqrt{-3}]$ and write 
$
\sigma(g)\cdot \lambda=\lambda+
\left[\begin{smallmatrix}m_1 & m_2 & m_3\\ n_1 & n_2 & n_3\end{smallmatrix}\right]
$.
Then $F_i(g\cdot(u,v))= A_i(g,u,v) \, F_i(u,v)$ with $A_i(g,u,v)$ given by
\[
\kappa(\sigma(g)) \, 
e^{\pi \sqrt{-1} \phi(\sigma(g), c_{1i})} \,
e^{-\frac{2 \pi \sqrt{-1}}{6} (2(n_1-n_3)k+n_2(2l+1))} \,
j_1(g,(u,v))j_2(g,(u,v))\, .
\]
\end{lemma}
This gives us  the transformation behavior of the $F_i$ under the 
generators $g_i$ of the group $\Gamma[\sqrt{-3}]$ given in 
Lemma \ref{generators}.
\begin{corollary}
We have ${F_i}_{|{1,1}} g_j=c(i,j) F_i$ with $c(i,j)$ given in the 
following diagram.
\end{corollary}

\[
\begin{tabular}{c|c|c|c|c|c|c}
$i \backslash j$ & 
$0$ & 
$1$ & 
$2$ & 
$3$ & 
$4$ & 
$5$ 
\\
\hline
$0$ & 
$1$ & $\rho$ & $\rho$ & $\rho^2$ & $1$ & $1$  
\\
\hline
$1$ & 
$1$ & $\rho$ & $1$ & $1$ & $\rho$ & $\rho^2$ 
\\
\hline
$2$ & 
$1$ & $1$ & $\rho$ & $\rho$ & $\rho$ & $\rho$
\\
\hline
$3$ & 
$1$ & $1$ & $\rho$ & $1$ & $\rho$ & $1$
\\
\hline
$4$ & 
$1$ & $\rho$ & $1$ & $1$ & $1$ & $1$
\end{tabular}
\]

\end{section}
\begin{section}{Construction of $\chi_{4,4}$ and $E_{1,1}$.}\label{Construction}
Finis constructed in \cite{Fi} elliptic modular functions $X,Y,Z$
for the elliptic curve ${\CC}/\Lambda$ 
with $\Lambda= \sqrt{-3}\, O_F$, that satisfy
the relation 
$$
X^3=\rho \, (Y^3-Z^3)\, .
$$
Here $X$, $Y$ and  $Z$ are normalized  theta functions 
satisfying
$$
f(z+\alpha)= {\rm exp}\left( \frac{2\pi}{\sqrt{3}} (\bar{\alpha} z - \rho N(\alpha))\right) f(z)
\qquad (\alpha \in \Lambda)
$$
and $Z(z)=Y(-z)$.
The zero divisor of $X$ is the degree $3$ divisor
$O_F \bmod \Lambda$  and that of $Y$ is  
$1/\sqrt{-3}+O_F \bmod \Lambda$.
These functions can be defined by
$$
X(z)=\frac{1}{c} e^{\pi z^2/\sqrt{3}} \vartheta_{\left[\begin{smallmatrix} 1/2 \\ 1/2 \end{smallmatrix}\right]}(-\rho^2,z),
\quad
Y(z)=\frac{1}{c} e^{\pi z^2/\sqrt{3}} \vartheta_{\left[\begin{smallmatrix} 1/6 \\ 1/6 \end{smallmatrix}\right]}(-\rho^2,z),
$$
with 
$c= \vartheta_{\left[ \begin{smallmatrix} 1/6 \\ 1/6 
\end{smallmatrix}\right]}(-\rho^2,0)$
as given in (13).
The functions $X$ and $Y$ satisfy  for $\alpha \in O_F$
$$
X(z+\alpha)=
e^{\frac{2\pi}{\sqrt{3}}(\bar{\alpha}z-\rho N(\alpha))}
X(z),\qquad
Y(z+\alpha)=
e^{\frac{2\pi}{\sqrt{3}}(\bar{\alpha}z-\rho N(\alpha))}\rho^{{\rm Tr}(\alpha)}
Y(z)\, .
$$
and $X(\epsilon z)=\epsilon X(z)$ for $\epsilon \in O_F^{\times}$ 
and $Y(\rho z)=Y(z)$.
Moreover, we have the identity
$$
X(\sqrt{-3} \, z)= \sqrt{-3}\,  X(z)Y(z)Z(z)\, .
$$

We can develop the theta functions $\vartheta_{\lambda}$ for
$\lambda \in C_1'$ in  Fourier-Jacobi series. 
Working this out and substituting this in the $F_i$ 
as defined in equation (12), one finds after some amount of calculation 
the following Fourier-Jacobi series of the $F_i$. 

We set 
$$\xi=(\rho^2-1)/3 \qquad \text{\rm and } \qquad  
q_v=e^{2\pi v/\sqrt{3}}
$$
and obtain
\begin{align*}
F_0(u,v)&=
\sum_{\alpha \in {O}_F}
\rho^{-Tr(\alpha)}
\left[\begin{smallmatrix} X'(\alpha u) \\ \frac{2\pi}{\sqrt{3}}\bar{\alpha}X(\alpha u) \end{smallmatrix}\right]
q_v^{N(\alpha)},\nonumber
\\
F_1(u,v)&=
\sum_{\alpha \in O_F+\xi}
\left[\begin{smallmatrix} X'(\alpha u) \\ \frac{2\pi}{\sqrt{3}}\bar{\alpha}X(\alpha u) \end{smallmatrix}\right]
q_v^{N(\alpha)},\nonumber
\\
F_2(u,v)&=
\sum_{\alpha \in {O}_F+\xi}
e^{\frac{2\pi}{\sqrt{3}}(\alpha\bar{\xi}-\bar{\alpha}{\xi})}
\left[\begin{smallmatrix} Y'(\alpha u) \\ \frac{2\pi}{\sqrt{3}}\bar{\alpha}Y(\alpha u) \end{smallmatrix}\right]
q_v^{N(\alpha)},
\\
F_3(u,v)&=
\sum_{\alpha \in {O}_F+\xi}
e^{\frac{2\pi}{\sqrt{3}}(\alpha\bar{\xi}-\bar{\alpha}\xi)}
\left[\begin{smallmatrix} Z'(\alpha u) \\ \frac{2\pi}{\sqrt{3}}\bar{\alpha}Z(\alpha u) \end{smallmatrix}\right]
q_v^{N(\alpha)},\nonumber
\\
F_4(u,v)&=
\sum_{\alpha \in {O}_F}
\left[\begin{smallmatrix} X'(\alpha u) \\ \frac{2\pi}{\sqrt{3}}\bar{\alpha}X(\alpha u) \end{smallmatrix}\right]
q_v^{N(\alpha)}. \nonumber
\end{align*}

Using the transformation behavior of $X,Y,Z$ one can calculate the
transformation of the $F_i$ under the generators $r_k$ for $k=1,2,3$
of the $\mathfrak{S}_4$-part, see Section \ref{sectionPMF}.
Putting all of that together, we get
\[
\begin{tabular}{c|c|c|c|c|c}
$i$ & 
$F_0\vert_{1,1} r_i^{-1}$ & 
$F_1\vert_{1,1} r_i^{-1}$ & 
$F_2\vert_{1,1} r_i^{-1}$  &
$F_3\vert_{1,1} r_i^{-1}$  &
$F_4\vert_{1,1} r_i^{-1}$  
\\
\hline
$1$ & 
$-F_1$ & 
$-F_0$ & 
$-F_2$ & 
$-F_3$ &
$-F_4$ 
\\
\hline
$2$ & 
$-F_0$ & 
$-F_1$ & 
$-F_3$ &
$-F_2$ &
$-F_4$ 
\\
\hline
$3$ & 
$F_0$ &
$e^{5 \pi i/9}F_2$ & 
$e^{-2 \pi i/3}F_3$ & 
$e^{\pi i/9}F_1$ &
$F_4$
\end{tabular}
\]
After these preparations we can construct the two basic modular forms.
We put 
$$
\chi_{4,4}={\rm Sym}^4(F_0,F_1,F_2,F_3)\quad 
\hbox{\rm  and } \quad  E_{1,1}=F_4 \, .
$$
\begin{corollary}
We have $\chi_{4,4} \in S_{4,4}(\Gamma[\sqrt{-3}],{\det}^2)$ 
and $E_{1,1} \in M_{1,1}(\Gamma[\sqrt{-3}],{\det})$.
\end{corollary}
\begin{remark}
Both $\chi_{4,4}$ and $E_{1,1}$ are Hecke eigenforms. 
But the eigenvalues of
$E_{1,1}$ are not always integral, see \cite[Remark 12.3]{C-vdG}.

In \cite[Section 11.3,Case 1]{B-vdG} it is conjectured that there is a
lift from $S^{-}_8(\Gamma_0(9))$ (see loc. cit., Section 11.1 for the
notation) to the $\mathfrak{S}_4$-invariant part of the space
$S_{4,4}(\Gamma[\sqrt{-3}],{\det}^2)$. 
The space $S_8^{-}(\Gamma_0(9))$ is one-dimensional and generated by one form, say $f$, whose Fourier expansion $f(\tau)= \sum_{n\geqslant 1}a_n(f)q^n$
can be normalised as
$$
f(\tau)= 
q+ 232\, q^4+260\, q^7-5760\, q^{10} +6890\, q^{13}+
7744\, q^{16}+33176\, q^{19} +\ldots\, .
$$
The claim that $\chi_{4,4}$ is a lift of $f$ is supported 
by the fact the Hecke eigenvalue
of $\chi_{4,4}$ at $1+3\rho$ (of norm $7$) is given by
\[
\lambda_{1+3\rho}(\chi_{4,4})=309-882\rho.
\]
This Hecke eigenvalue is computed directly by using the Fourier-Jacobi expansion of the last component of
$\chi_{4,4}$ (recall that the last component of a vector-valued Picard modular Hecke  eigenform is sufficient
to compute its Hecke eigenvalues). Since the Fourier-Jacobi expansion 
of $\chi_{4,4}^{(4)}$ (after suitable normalisation) 
starts with $X^2q_v^2$, 
for computing its Hecke eigenvalue at $1+3\rho$, 
we need its Fourier-Jacobi coefficient at $q_v^{14}$.
This Fourier-Jacobi coefficient 
will be  given in the next section and this gives
$$
\lambda_{1+3\rho}(\chi_{4,4})=309-882\rho=260+(1+3\rho^5)(1+3\rho^2)^2=a_7(f)+(1+3\rho^5)(1+3\rho^2)^2.
$$
\end{remark}
\end{section}
\begin{section}{The Fourier-Jacobi expansion of $\chi_{1,1}$}
Here we develop the form $E_{1,1}$ in a Fourier-Jacobi series. 
Recall the definition of $E_{1,1}$ with $q_v=e^{2\pi v/\sqrt{3}}$
\[
E_{1,1}(u,v)=
\sum_{\alpha \in {O}_F}
\left[\begin{smallmatrix} X'(\alpha u) \\ \frac{2\pi}{\sqrt{3}}\bar{\alpha}X(\alpha u) \end{smallmatrix}\right]
q_v^{N(\alpha)}.
\]
Using the relation $X(\epsilon u)=\epsilon X(u)$ for 
$\epsilon \in O_F^{\times}$, we get
\[
E_{1,1}(u,v)=
\left[\begin{smallmatrix} X'(0) \\ 0 \end{smallmatrix}\right]
+
6\left[\begin{smallmatrix} X'(u) \\ \frac{2\pi}{\sqrt{3}}X(u) \end{smallmatrix}\right]\,q_v
+
6\left[\begin{smallmatrix} X'(u\sqrt{-3}) \\ -\frac{2\pi}{\sqrt{3}}\sqrt{-3}X(u\sqrt{-3}) \end{smallmatrix}\right]\,q_v^3+\ldots
\]
and using the Shintani operators (see \cite[\S 4]{C-vdG}), we can rewrite this as
\[
E_{1,1}(u,v)=
\left[\begin{smallmatrix} X'(0) \\ 0 \end{smallmatrix}\right]
+
6\left[\begin{smallmatrix} X'(u) \\ \frac{2\pi}{\sqrt{3}}X(u) \end{smallmatrix}\right]\,q_v
+
6\left[\begin{smallmatrix} (XYZ)'(u) \\ \frac{2\pi}{\sqrt{3}}3XYZ(u) \end{smallmatrix}\right]\,q_v^3+\ldots
\]
We set
$$
P_n(X,Y,Z)=\sum_{\alpha \in N_n} \bar{\alpha}X(\alpha u)\, ,
$$
where we use the notation 
$$
N_{n}=\{\alpha=a+b\rho \in O_F\, | \, N(\alpha)=a^2-ab+b^2=n\}.
$$
If $P'_n(X,Y,Z)$ denotes the derivative of $P_n(X,Y,Z)$ with
respect to the variable $u$
\[
P'_n(X,Y,Z)=(\sum_{\alpha \in N_n}
\bar{\alpha}X(\alpha u))'=
\sum_{\alpha \in N_n}
\alpha\bar{\alpha}X'(\alpha u)
=n\sum_{\alpha \in N_n}
X'(\alpha u)\, ,
\]
we have
\[
E_{1,1}(u,v)=
\left[\begin{smallmatrix} X'(0) \\ 0 \end{smallmatrix}\right]
+\sum_{n \geqslant 1}
\left[\begin{smallmatrix} P'_n(X,Y,Z)/n \\ \frac{2\pi}{\sqrt{3}}P_n(X,Y,Z) \end{smallmatrix}\right]q_v^n \, .
\]
In order to get more terms in the Fourier-Jacobi expansion of $E_{1,1}$
we write the set $N_n$ as
$
N_n=
\alpha_1\, O_F^{\times} \sqcup \alpha_2\, O_F^{\times} 
\sqcup \ldots \sqcup \alpha_j\, O_F^{\times}
$
and split $P_n$ according to this decomposition:
\[
P_n(X,Y,Z)=\sum_{\alpha \in N_n}
\bar{\alpha}X(\alpha u)
=\sum_{i=1}^j
\sum_{\epsilon \in O_F^{\times}}
\bar{\epsilon}\bar{\alpha_i}X(\epsilon \alpha_i u)
=
\sum_{i=1}^j
\sum_{\epsilon \in  O_F^{\times}}
\bar{\alpha_i}X(\alpha_i u)
=
6\, \sum_{i=1}^j \bar{\alpha_i}X(\alpha_i u)
\]
The polynomials $P_n$ are homogeneous of degree $n$ in $X, Y$ and $Z$ 
and the first few of them are given by
\begin{align*}
P_1&=6\,X;  \quad P_3=18\,XYZ; \quad  P_4=12\,X(Y^3+Z^3);\\
P_7&=-6\,X(Y^6-16Y^3Z^3+Z^6); \quad  P_9=54\, XYZ(Y^6-Y^3Z^3+Z^6);\\
P_{12}&=-36\, XYZ(2Y^{9}-3Y^6Z^3-3Y^3Z^6+2YZ^{9});\\
P_{13}&=6\,X(5Y^{12}-7Y^9Z^3+30Y^6Z^6-7Y^3Z^9+5Z^{12})\, .\\
\end{align*}
\end{section}
\begin{section}{The Fourier-Jacobi expansion of $\chi_{4,4}$}
The Fourier-Jacobi expansion
of $\chi_{4,4}$ is determined by those of $F_0,\ldots,F_3$.
We determine these and start with $F_0$.
We set
$$
Q_n(X,Y,Z)=\sum_{\alpha \in N_n}
\rho^{-Tr(\alpha)} \bar{\alpha}X(\alpha u) \, ,
$$
which gives for the derivative of $Q_n(X,Y,Z)$ with respect to the variable $u$
$$
Q'_n(X,Y,Z)=n\sum_{\alpha \in N_n}
\rho^{-Tr(\alpha)}X'(\alpha u)
$$
and this leads to the expansion
$$
F_{0}(u,v)=
\sum_{\alpha \in O_F}
\rho^{-Tr(\alpha)}
\left[\begin{smallmatrix} X'(\alpha u) \\ \frac{2\pi}{\sqrt{3}}\bar{\alpha}X(\alpha u) \end{smallmatrix}\right]
q_v^{N(\alpha)}
=
\left[\begin{smallmatrix} X'(0) \\ 0 \end{smallmatrix}\right]
+\sum_{n \geqslant 1}
\left[\begin{smallmatrix} Q'_n(X,Y,Z)/n \\ \frac{2\pi}{\sqrt{3}}Q_n(X,Y,Z) \end{smallmatrix}\right]q_v^n \, .
$$
\begin{lemma}
We have $Q_n=P_n$ if $ n\equiv \, 0 \, (\bmod \, 3)$, else $Q_n=-P_n/2$.
\end{lemma}
\begin{proof} Using $X(\epsilon u)=\epsilon X(u)$ and the decomposition
$N_n=\sqcup_{i=1}^j \alpha_i \, O_F^{\times}$ as above we get
\[
Q_n(X,Y,Z)=
\sum_{i=1}^j
\big(\sum_{\epsilon \in O_F^{\times}}
\rho^{-Tr(\epsilon\alpha_i)} \big)
\bar{\alpha_i}X(\alpha_i u)
\]
and writing $\alpha_i=a_i+\rho b_i$ with $a_i, b_i \in \ZZ$, we have
\[
\sum_{\epsilon \in O_F^{\times}}
\rho^{-Tr(\epsilon\alpha_i)}=3(\rho^{a_i+b_i}+\rho^{2(a_i+b_i)})=
\begin{cases}
6 \,\,\,\,  \text{ if } a_i+b_i\equiv 0\, (\bmod\, 3) \\
-3 \text{ if } a_i+b_i\not\equiv 0\,  (\bmod\, 3)\, .
\end{cases}
\]
Noticing that $N(\alpha_i)\equiv (a_i+b_i)^2\, (\bmod\, 3) $, we get 
the desired result.
\end{proof}
For $F_1$ we set
$
N_n(\xi)=\xi\cdot\{\alpha\in O_F\, | \, N(\alpha)=n, 
\alpha\equiv 1 \, (\bmod \, \sqrt{-3})\}
$
and note that the map $\alpha+\xi\mapsto \xi(\alpha(\rho-1)+1)$ 
is a bijection from $
\{\alpha+\xi \, | \, \alpha\in O_F, N(\alpha+\xi)=n/3\}
$ to $N_n(\xi)$.
We define
\[
R_n(X,Y,Z)=
\sum_{\alpha \in N_n(\xi)}
\bar{\alpha}X(\alpha u)\, ,
\]
so we have
\[
R'_n(X,Y,Z)=\frac{n}{3}
\sum_{\alpha \in N_n(\xi)}
X'(\alpha u)\, ,
\]
and we can write
\[
F_1(u,v)=
\sum_{\alpha \in O_F+\xi}
\left[\begin{smallmatrix} X'(\alpha u) \\ \frac{2\pi}{\sqrt{3}}\bar{\alpha}X(\alpha u) \end{smallmatrix}\right]
q_v^{N(\alpha)}=
\sum_{n \geqslant 1}
\left[\begin{smallmatrix} 3\,R'_n(X,Y,Z)/n \\ \frac{2\pi}{\sqrt{3}}R_n(X,Y,Z) \end{smallmatrix}\right]q_v^{n/3}\, .
\]
We relate the polynomials  $R_n$ and $P_n$. For $n\equiv 0 \, (\bmod \, 3)$ we have
$R_n=0$; if $n\equiv 1\, (\bmod\, 3)$ we write 
$$
N_n(\xi)= \sqcup_{i=1}^j \xi \alpha_i \{ 1,\rho,\rho^2\}
$$
with $\alpha_i \equiv 1 (\bmod \sqrt{-3})$. We thus get 
$$
R_n(X,Y,Z)=\bar{\xi}P_n(X_0,Y_0,Z_0)/2 \qquad
\text{\rm 
with $X_0=X(\xi u), Y_0=Y(\xi u)$ and $Z_0=Z(\xi u)$.}
$$
For $F_2$ we use 
$$
S_n(X,Y,Z)=
\sum_{\alpha \in N_n(\xi)}
e^{\frac{2\pi}{\sqrt{3}}(\alpha\bar{\xi}-\bar{\alpha}\xi)}
\bar{\alpha}Y(\alpha u) \, ,
$$
so we obtain
$$
S'_n(X,Y,Z)=\frac{n}{3}
\sum_{\alpha \in N_n(\xi)}
e^{\frac{2\pi}{\sqrt{3}}(\alpha\bar{\xi}-\bar{\alpha}\xi)}
Y'(\alpha u) \, .
$$ 
We then have
$$
F_2(u,v)=
\sum_{\alpha \in O_F+\xi}
e^{\frac{2\pi}{\sqrt{3}}(\alpha\bar{\xi}-\bar{\alpha}\xi)}
\left[\begin{smallmatrix} Y'(\alpha u) \\ \frac{2\pi}{\sqrt{3}}\bar{\alpha}Y(\alpha u) \end{smallmatrix}\right]
q_v^{N(\alpha)}=
\sum_{n \geqslant 1}
\left[\begin{smallmatrix} 3\,S'_n(X,Y,Z)/n \\ \frac{2\pi}{\sqrt{3}}S_n(X,Y,Z) \end{smallmatrix}\right]q_v^{n/3}.
$$

Writing again $N_n(\xi)=\sqcup_{i=1}^j \xi \alpha_i \, \{1,\rho,\rho^2\}$
with $\alpha_i-1\in (3)$ we find 
$$
S_n(X,Y,Z)=3 \, \xi \, \sum_{i=1}^j \bar{\alpha}_i\, Y(\alpha_i\xi u)\, .
$$
Using the Shintani operators we find the first few $S_n$, where again we
use the notation 
$X_0=X(\xi u)$, $Y_0=Y(\xi u)$ and $Z_0=Z(\xi u)$:
\begin{align*}
S_1&=3\,\bar{\xi}Y_0,\qquad
S_4=-6\,\bar{\xi}Y_0(-Y_0^3+2\, Z_0^3),\qquad
S_7=-3\, \bar{\xi}Y_0(Y_0^6+14\, Y_0^3Z_0^3-14\, Z_0^6),\\
S_{13}&=3\, \bar{\xi}
Y_0(5\, Y_0^{12}-13\, Y_0^9Z_0^3+39\, Y_0^6Z_0^6-52\, Y_0^3Z_0^9+26\, Z_0^{12})\, . \\
\end{align*}

Finally for $F_3$ we use 
$
F_3=-F_2\vert_{1,1}r^{-1}_2
$
and thus put $T_n(X,Y,Z)=S_n(X,Z,Y)$ and then can write
\[
F_3(u,v)=-
\sum_{\alpha \in O_F+\xi}
e^{\frac{2\pi}{\sqrt{3}}(\alpha\bar{\xi}-\bar{\alpha}\xi)}
\left[\begin{smallmatrix} Z'(\alpha u) \\ \frac{2\pi}{\sqrt{3}}\bar{\alpha}Z(\alpha u) \end{smallmatrix}\right]
q_v^{N(\alpha)}=-
\sum_{n \geqslant 1}
\left[\begin{smallmatrix} 3\,T'_n(X,Y,Z)/n \\ \frac{2\pi}{\sqrt{3}}T_n(X,Y,Z) \end{smallmatrix}\right]q_v^{n/3}.
\]
After these preparations we can calculate the beginning of the
Fourier-Jacobi expansion of $\chi_{4,4}$. The result is, after suitable normalisation,
\[
\chi_{4,4}=
\left[
\begin{matrix}
(\frac{\sqrt{3}}{2\pi})^4(3 \, c_1\rho^2(X'_0Y'_0Z'_0)\, q_v+O(q_v^2))\\
(\frac{\sqrt{3}}{2\pi})^3(c_1(2+\rho)(X_0Y'_0Z'_0+X'_0Y'_0Z_0+X_0Y'_0Z'_0)
\, q_v+O(q_v^2))\\
(\frac{\sqrt{3}}{2\pi})^2(-c_1X'q_v+O(q_v^2))\\
\frac{\sqrt{3}}{2\pi}(-\frac{c_1}{3}X q_v +4XX'q_v^2+O(q_v^4))\\
X^2q_v^{2}-6\,X^2YZq_v^{4} +O(q_v^5)
\end{matrix}
\right] \, ,
\]
where $c_1=X'(0)$. For the last coordinate one can calculate 
more terms. Indeed, this involves only the second coordinates
of the $F_i$ ($i=0,\ldots,3$) and one finds:
\begin{align*}
\chi_{4,4}^{(4)}=&X^2\bigg(
q_v^{2}-6\,YZq_v^{4} -20\,(Y^{3}+Z^{3})q_v^{5}+81\,Y^{2}Z^{2}q_v^{6}
+132\,(Y^{4}Z+YZ^{4}) q_v^{7}\\
&+(122\,Y^{6}-800\,Y^{3}Z^{3}+122\,Z^{6})q_v^{8}+
(-1020\,Y^{7}Z+1470\,Y^{4}Z^{4}-1020\,Y{Z}^{7} )q_v^{10}\\
&+( -76\,Y^{9}+1140\,Y^{6}Z^{3}+1140
\,Y^{3}Z^{6}-76\,Z^{9} ) q_v^{11}\\
& + ( -486\,Y^{8}Z
^{2}+486\,Y^{5}Z^{5}-486\,Y^{2}Z^{8} ) q_v^{12}\\
&+ (
3012\,Y^{10}Z-3924\,Y^{7}Z^{4}-3924\,Y^{4}Z^{7}+3012\,YZ^{10} )q_v^{13}\\
&+ ( -1261\,Y^{12}+266\,Y^{9}Z^{3}+4782
\,Y^{6}Z^{6}+266\,Y^{3}Z^{9}-1261\,Z^{12} ) q_v^{14}+ \ldots \bigg) \, . \\
\end{align*}
\begin{remark}\label{thetabasis} 
To identify the terms in the Fourier-Jacobi expansions we
use the fact that we know a basis for the space of
theta functions of degree $3n$
on the elliptic curve ${\CC}/\Lambda$:
$$
\{ X^aY^bZ^c: 0\leq a \leq 2, \, 0 \leq b \leq n-a, \, a+b+c=n\}\, .
$$
We can use the Taylor expansion of an element in this space around the origin
to express it in terms of such a basis.
\end{remark}
\end{section}
\begin{section}{Restriction to the curve $T_1$}\label{Restriction-section}
In order to know which covariants yield holomorphic modular forms
we need the expansion of the modular forms $\zeta$, 
$\chi_{4,4}$ and $E_{1,1}$ along $T_1$, given by $u=0$, 
the zero locus of $\zeta$.

We start with the expansions of the elliptic 
functions $X$, $Y$ and $Z$ near the origin of the
elliptic curve ${\CC}/\Lambda$. These have the form
$$
X(z)=\sum_{j\geqslant 0} c_{6j+1}z^{6j+1},\quad
Y(z)=\sum_{j\geqslant 0} d_{3j}z^{3j},\quad
Z(z)=\sum_{j\geqslant 0} (-1)^j\,d_{3j}z^{3j}. \eqno(14)
$$
Note that the functions $Y$ and $Z$ are normalised such that $d_0=1$.
Let $\xi=(\rho^2-1)/3=\rho\sqrt{-3}/3$ as before.
By \cite[Lemma 9, formulas (68--69)]{Fi} these functions satisfy
\[
Y^3(\xi z)=\frac{1}{\rho-1}(\rho Y(z)-Z(z)) \quad \text{and} \quad
Z^3(\xi u)=\frac{1}{\rho-1}(-Y(z)+\rho Z(z)) \, ,
\]
and with  $X^3=\rho(Y^3-Z^3)$ , we have
$
X^3(\xi z)=-\xi \, (Y(z)-Z(z)).
$
This relation provides links between the numbers $c_{6j+1}$ and $d_{6j+3}$, while
the relation $Y^3(\xi z)+Z^3(\xi z)=Y(z)+Z(z)$
provides the relation between the
$c_{6j+1}$ and $d_{6j}$. 
If follows that the numbers
$c_{6j+1}$ and $d_{3j}$ can be expressed in terms of powers of $c_1$.
For example, we have:
\begin{align*}
&c_7= 6\, \rho \, c_1^7/7!, \quad
c_{13} = -6^3\rho^2\, c_1^{13}/13!, \quad
c_{19}=-2^7 3^4 5\cdot 23\, c_1^{19}/19!,\\
&d_3= \rho^2 c_1^3/3!, \quad
d_6 = -2\, \rho c_1^6/6!, \quad
d_9 = -8\, c_1^9/9!, \quad
d_{12}= -2^3 \, 19\,\rho^2c_1^{12}/12!, \\
&d_{15}=2^3 5\cdot 31\, \rho \, c_1^{15}/15! \, .
\end{align*}
Here the constant $c_1$ is given by
$$
c_1=\Gamma(1/3)^{3}\,  e^{-17i \pi/18}/(2\, \pi)\, . \eqno(15)
$$
\begin{remark}
These numbers are related to the development of the modular form
$\vartheta$ around its zero $\tau_0=(1-\rho^2)/3$, see 
\cite[Prop.\ 17]{Zagier}.
\end{remark}
For the restriction to the 
curve $T_1$ of the modular forms  $\zeta$, $E_{1,1}$ and
$\chi_{4,4}$
we can apply Proposition \ref{quasi-modularity}.

The definition of the cusp form 
$\zeta \in S_{6}(\Gamma[\sqrt{-3},\det])$ yields by (14)
with $q_v=e^{2\pi \, v/\sqrt{3}}$
\[
\zeta(u,\sqrt{-3}\tau)=1/6\, \sum_{\alpha \in O_F}
\alpha^5X(\alpha u)q^{N(\alpha)}=
1/6\,\sum_{j\geqslant 0}
a_{6j+1}\Theta_{6(j+1)}(\tau)u^{6j+1}\, ,
\]
where
$\Theta_j(\tau)= \sum_{\alpha \in O_F}\alpha^jq^{N(\alpha)}\in M_{j+1}(\Gamma_1(3))$, as introduced in Section \ref{ModularCurves}.
Therefore with $w=c_{1} u$  the Taylor expansion of $\zeta$ about $u=0$ starts with
\begin{align*}
\zeta(u,\sqrt{-3}\tau)&=c_{1}\,\vartheta\eta^8\psi^2(\tau)u+
c_{7}\,\vartheta \eta^{8} \left(\eta^{16}+18 \psi^{4} \eta^{8}+729 \psi^{8}\right)u^7+\ldots\\
&=\vartheta\eta^8(\tau)
\bigg(\psi^2(\tau)w+
(\rho/840)\,(\eta^{16}+18 \psi^{4} \eta^{8}+729 \psi^{8})(\tau)w^7+\ldots
\bigg)\, .
\end{align*}
In a similar way, we obtain the development along the  curve $T_1$ 
of the Eisenstein series $E_{1,1}$:
\begin{align*}
E_{1,1}(u,\sqrt{-3}\tau)&=
\left[\begin{smallmatrix}
E^{(0)}_{1,1}(u,\sqrt{-3}\tau)\\
E^{(1)}_{1,1}(u,\sqrt{-3}\tau)
 \end{smallmatrix}\right]
=
\sum_{\alpha \in \mathcal{O}}
\left[\begin{smallmatrix} X'(\alpha u) \\ \frac{2\pi}{\sqrt{3}}\bar{\alpha}X(\alpha u) \end{smallmatrix}\right]
q^{N(\alpha)}\\
&=
\sum_{j\geqslant 0}
c_{6j+1}
\Big(
\left[\begin{smallmatrix}
(6j+1)\Theta_{6j}(\tau)\\
0
 \end{smallmatrix}\right]
u^{6j}+
\left[\begin{smallmatrix}
0\\
\frac{\Theta'_{6j}(\tau)}{\sqrt{-3}}
 \end{smallmatrix}\right]
u^{6j+1}
\Big)\, ,\\
\end{align*}
where the components $E^{(i)}_{1,1}$ have the following expansions in which
  the variable $\tau$ is omitted: 
\begin{align*}
E^{(0)}_{1,1}(u,\sqrt{-3}\tau)&=
c_1
\big(
\vartheta+
\rho\,\vartheta\psi^{2} \eta^{8}/20\, w^6+\ldots
\big)\\
E^{(1)}_{1,1}(u,\sqrt{-3}\tau)&=\frac{\pi}{6\sqrt{3}}
\big(
(108 \psi^{3}+\vartheta(e_2- \vartheta^2))w+
(\frac{\rho}{140})\,
\psi  \,\eta^{8}(5 \eta^{8}+243 \psi^{4}+7 e_2 \psi  \vartheta)\, w^7+\ldots
\big)
\end{align*}
For the cusp form $\chi_{4,4}$, we get
\[
\chi_{4,4}(u,\sqrt{-3}\tau)=
\left[\begin{smallmatrix} 
0 \\ 0 \\ h_0 \\ 0 \\ 0
 \end{smallmatrix}\right] +
\left[\begin{smallmatrix} 
0 \\ 0 \\ 0 \\ h_1 \\ 0
 \end{smallmatrix}\right] u+
 \left[\begin{smallmatrix} 
0 \\ 0 \\ 0 \\ 0 \\ h_2
 \end{smallmatrix}\right] u^2+
 \left[\begin{smallmatrix} 
h_4 \\ 0 \\ 0 \\ 0 \\ 0
 \end{smallmatrix}\right] u^4+
\left[\begin{smallmatrix} 
0 \\ h_5 \\ 0 \\ 0 \\ 0
 \end{smallmatrix}\right] u^5 +\ldots     \, ,
\]
where $h_0$ and $h_4$ are cusp forms on $\Gamma_1(3)$, 
while $h_1$, $h_2$ and $h_5$
are quasi-modular forms on $\Gamma_1(3)$. We set 
$$
\gamma= 2\pi/\sqrt{3}\, . \eqno(16)
$$ 
Then the $h_i$  are given by
\begin{align*}
h_0&=-\frac{c_1^2}{\gamma^2}\,\eta^8\psi^2; \quad
h_1=-\frac{c_1^2}{6\,\gamma}\,\eta^8\psi^2(\vartheta^2+e_2);\quad
h_2=\frac{c_1^2}{144}\, \eta^8\psi^2(3\, \vartheta^2+e_2)(\vartheta^2-e_2)\\
h_4&=-\rho\, \frac{c_1^8}{12\,\gamma^4} \eta^8\psi^2\vartheta^2; \quad
h_5=\rho\, \frac{c_1^8}{180\,\gamma^3} \eta^8\psi^2\vartheta(4\, \vartheta^3+5\, \vartheta e_2+54\, \psi^3).
\end{align*}
\begin{remark}
The cusp form $h_0$ is proportional to the cusp form
$f_6(\tau)=\eta^6(\tau)\eta^6(3\tau)$ of weight $6$
on $\Gamma_0(3)$ with Fourier expansion
\[
f_6(\tau)=q-6\, q^2+9\, q^3+4\, q^4+6\, q^5-54\, q^6-40\, q^7+168\, q^8+\ldots \, .
\]
This cusp form is one of the famous eta products and plays a similar role for $\Gamma_0(3)$
as the discriminant form $\Delta$ for ${\rm SL}(2,\ZZ)$.
\end{remark}

From the above expansions
we derive the order of vanishing along $T_1$ of $\chi_{4,4}$ and $\chi_{1,1}$.
This is used in Corollary \ref{orders}.

\begin{corollary}\label{orders2}
The order  of the five coordinates of $\chi_{4,4}$ along $u=0$ is $(4,5,0,1,2)$.
The order of the two coordinates of $E_{1,1}$ along $u=0$ is $(0,1)$.
\end{corollary}

\end{section}
\begin{section}{Construction of Modular Forms from Invariants}
In this section we shall use the  map $\nu : 
\mathcal{C}(V_4\oplus V_1) \to \mathbb{M}[1/\zeta]$ to construct modular
forms. 
Under $\nu$ a covariant $J_{a,b,c}$ 
of degree $(a,b,c)$ in the variables $(a_i,b_i,x_i)$ 
maps to a meromorphic modular form of weight
$(c,(3b-c)/2)$ on $\Gamma[\sqrt{-3}]$ with character 
$\epsilon^{a+b}\circ \text{det}^{2a+2b+2c}$, that is,
$$
\nu(J_{a,b,c}) \in \widetilde{M}_{c,3b/2,-a-b}(\Gamma[\sqrt{-3}])
$$
with the property that
$$
\nu(J_{a,b,c}) \quad \text{ is }
\begin{cases}
\mathfrak{S}_4\text{-invariant if } a+b\equiv 0 \bmod 2\\
\mathfrak{S}_4\text{-anti-invariant if } a+b\equiv 1 \bmod 2\, .
\end{cases}
$$
Here the tilda on $M$ refers to the meromorphicity along $T_1$.
In the following table we give for the twenty generating covariants
$J_{a,b,c}$ the weight $(j,k,l)$, the index $e=a+b (\bmod 2)$ and the
order of the coordinates of the meromorphic modular form
$\nu(J_{a,b,c})$ along $T_1$.

\begin{footnotesize}
\smallskip
\vbox{
\bigskip\centerline{\def\quad{\hskip 0.6em\relax}
\def\quod{\hskip 0.5em\relax }
\vbox{\offinterlineskip
\hrule
\halign{&\vrule#&\strut\quod\hfil#\quad\cr
height2pt&\omit&&\omit&&\omit&&\omit&&\omit&&\omit&\cr
&$(a,b,c)$ && $j$ && $k$ && $l$ && $e$ && order along $T_1$ &\cr
\noalign{\hrule}
&$(2,0,0)$ && $0$ && $0$ && $1$ && $0$ && $-2$ & \cr
&$(0,1,1)$ && $1$ && $1$ && $1$ && $1$ && $[0,1]$ & \cr
&$(3,0,0)$ && $0$ && $0$ && $0$ && $1$ && $-3$ & \cr
&$(1,4,0)$ && $0$ && $6$ && $1$ && $1$ && $1$ & \cr
&$(1,3,1)$ && $1$ && $4$ && $1$ && $0$ && $[0,1]$ & \cr
&$(1,2,2)$ && $2$ && $2$ && $1$ && $1$ && $[-1,0,1]$ & \cr
&$(1,1,3)$ && $3$ && $0$ && $1$ && $0$ && $[4,-1,0,1]$ & \cr
&$(1,0,4)$ && $4$ && $-2$ && $1$ && $1$ && $[3,4,-1,0,1]$ & \cr
&$(2,4,0)$ && $0$ && $6$ && $0$ && $0$ && $0$ & \cr
&$(2,3,1)$ && $1$ && $4$ && $0$ && $1$ && $[-1,0]$ & \cr
&$(2,2,2)$ && $2$ && $2$ && $0$ && $0$ && $[-2,-1,0]$ & \cr
&$(2,1,3)$ && $3$ && $0$ && $0$ && $1$ && $[3,-2,-1,0]$ & \cr
&$(2,0,4)$ && $4$ && $-2$ && $0$ && $0$ && $[2,3,-2,-1,0]$ & \cr
&$(3,6,0)$ && $0$ && $9$ && $0$ && $1$ && $0$ & \cr
&$(3,5,1)$ && $1$ && $7$ && $0$ && $0$ && $[-1,0]$ & \cr
&$(3,4,2)$ && $2$ && $5$ && $0$ && $1$ && $[4,-1,0]$ & \cr
&$(3,3,3)$ && $3$ && $3$ && $0$ && $0$ && $[3,4,-1,0]$ & \cr
&$(3,2,4)$ && $4$ && $1$ && $0$ && $1$ && $[2,3,4,-1,0]$ & \cr
&$(3,1,5)$ && $5$ && $-1$ && $0$ && $0$ && $[1,2,3,4,-1,0]$ & \cr
&$(3,0,6)$ && $6$ && $-3$ && $0$ && $1$ && $[6,1,2,3,4,-1,0]$ & \cr
} \hrule}
}}
\end{footnotesize}

\medskip

As we saw in Section \ref{sectionPMF} 
we have $M(\Gamma)=\CC[E_6,E_{12},E_9^2]$ and 
 first cusp form appears in weight 12 and is given by
$$ 
\chi_{12}=(E_6^2-E_{12})/5184\, .
$$
The cusp form $\chi_{12}$ is a Kudla lift of 
an element in $S^{-}_{11}(\Gamma_1(3))$ 
(see \cite[Prop.\ 10]{Fi}, or 
\cite[Section 11.3, Case 2b]{B-vdG}).
Moreover, there is a cusp form 
$$ 
\chi_{18}= (E_6^3-E_9^2)/3888\, .
$$

By calculating the expansions for some of these
$\nu(J_{a,b,c})$ one can identify sometimes the resulting modular forms.
We use here what was mentioned in Remark \ref{thetabasis}

Doing this for the generators $J_{a,b,0}$
one obtains the 
following proposition. Recall $\gamma=2 \pi/\sqrt{3}$ and $c_1=X'(0)$
as given in equations (15) and (16).
\begin{proposition}\label{nu-images}
The images under $\nu$ of the generators $J_{a,b,0}$ are;
\begin{align*}
\nu(J_{1,4,0})= \frac{3\, c_1^4}{70} \zeta \, , & \qquad
\nu(J_{2,0,0})= \frac{c_1^4}{6\gamma^4} \frac{\chi_{12}}{\zeta^2} \, , \qquad
\nu(J_{3,0,0})= \frac{c_1^6 }{864\, \gamma^6} \frac{\chi_{18}-E_6\chi_{12}}{\zeta^3}\, , \\
\nu(J_{2,4,0})&= -\frac{c_1^6}{5040\, \gamma^2} \, E_6 \, , \qquad
\nu(J_{3,6,0})= \frac{c_1^9}{798336 \, \gamma^3}\, E_9 \, . \\
\end{align*}
\end{proposition}

\begin{remark}
As a check, one may apply 
$\nu$ to the relation 
$$
3J_{2,4,0}^3 + \frac{13068}{875}\, J_{3,6,0}^2 + 
J_{1,4,0}^3J_{3,0,0} - \frac{3}{2}\,J_{2,4,0}J_{1,4,0}^2J_{2,0,0}=0. \eqno(17)
$$
(see (6) in Section \ref{covariants})
and obtain
$$
\frac{c_1^{18}}{42674688000 \,\gamma^6}(-E_6^3 + E_9^2 + 3888\chi_{18})=0\, , \eqno(18)
$$
in agreement with the definition of $\chi_{18}$.
\end{remark}

\begin{remark}
The image of the discriminant 
$32(J_{2,0,0}^3-6\, J_{3,0,0}^2)$ of the quartic polynomial $f_4$ 
under $\nu$ is constant:
$$
\nu(32(J_{2,0,0}^3-6\, J_{3,0,0}^2))=-\frac{\rho\,  c_1^{12}}{3^3\, \gamma^{12}}\, .
$$
This comes about by the fact that the moduli space is obtained
by blowing up of the discriminant locus, cf.\ the diagram in
\cite[p.\ 6]{CFG2} and the discussion there.
\end{remark}
We finish this section with a result on the module of scalar-valued cusp
forms on $\Gamma$.
Since the group $\Gamma$ has a unique cusp, we have
$ \dim S_k(\Gamma)= \dim M_k(\Gamma) -1$ if $\dim M_k(\Gamma)>0$
and the generating series for the dimension of the spaces 
$S_k(\Gamma)$ is given by
\[
\sum_{k\geqslant 0} \dim S_k(\Gamma) t^k=
\frac{t^{12}+t^{18}-t^{30}}{(1-t^6)(1-t^{12})(1-t^{18})}=
 t^{12}+2\, t^{18}+3\, t^{24}+4\, t^{30}+6\, t^{36}+\ldots
\]
The first cusp form appears in weight 12, namely 
$\chi_{12}$.
Then we have
\begin{align*}
S_{18}(\Gamma)&=\text{Span}_{\CC}(E_6\chi_{12},\chi_{18}),\quad
S_{24}(\Gamma)=\text{Span}_{\CC}(E_6^2\chi_{12},E_6\chi_{18},\chi_{12}^2),\\
S_{30}(\Gamma)&=
\text{Span}_{\CC}(E_6^3\chi_{12},E_9^2\chi_{12},\chi_{12}\chi_{18},\chi_{12}^2,E_6^2\chi_{18},E_6^2\chi_{12}^2
)\, ,
\end{align*}
but in the last case, we have the relation
$
\chi_{12}(3888\, \chi_{18}-(E_6^3-E_9^2))=0
$
which comes from the relation (18) multiplied by 
$J_{1,4,0}^2J_{2,0,0}$ which corresponds to $\chi_{12}$. 
Here this relation actually counts as a relation between cusp forms.
\begin{corollary}
The $M(\Gamma)$-module ${\Sigma}(\Gamma)=\oplus_k S_k(\Gamma)$ 
of cusp forms on $\Gamma$ is generated by the forms $\chi_{12}$
and $\chi_{18}$ with the relation $3888\, \chi_{18}\chi_{12}-(E_6^3-E_9^2)\chi_{12}=0$ in weight $30$.
\end{corollary}

\end{section}
\begin{section}{The structure of a module for $j=4$}
To show the feasibility of constructing modular forms by covariants,
as an application we determine the structure of the $M(\Gamma)$-module
$$
\mathcal{M}_4^{2}(\Gamma)=
\oplus_{k\geqslant 0} M_{4,k}(\Gamma,\text{det}^2)\, .
$$
With the same method one can also treat the modules
$$
\mathcal{M}_4^{l}(\Gamma)=
\oplus_{k\geqslant 0} M_{4,k}(\Gamma,\text{det}^l)
$$
for $l=0$ and $l=1$, but we refrain from giving the details.

The structure of the modules 
$\mathcal{M}_j^l(\Gamma)= \oplus_k M_{j,k}(\Gamma,\text{det}^l)$
for $j<4$ was determined in \cite{C-vdG} in a different manner, 
but it would be very difficult to go beyond these cases that way.
Invariant theory provides a good way to build generators.

For the cusp forms, we use the notation
$\Sigma_4^l(\Gamma)$ or $\Sigma_4^l(\Gamma[\sqrt{-3}])=
\oplus_{k\geqslant 0} 
S_{4,k}(\Gamma[\sqrt{-3}],\text{det}^l)$.
Note (see \cite[Proposition 5.1]{C-vdG})
that 
$$
\mathcal{M}_4^{l}=\Sigma_4^l \quad \text{\rm  if $l \not \equiv 1 \bmod 3$.}
$$ 
Recall that if $k\not \equiv 1 \bmod 3$ then
$M_{4,k}(\Gamma[\sqrt{-3}],\text{det}^l)=(0)$.

\begin{theorem}\label{Thm-Sigma-4-2}
The $M(\Gamma)$-module $\Sigma^2_4$ is freely generated by
cusp forms of weight $(4,4)$, $(4,10)$, $(4,16)$, $(4,22)$
and $(4,28)$.
\end{theorem}

\begin{proof}
We begin the proof by deducing the Hilbert--Poincar\'e series for
the module $\Sigma_4^2$. For this we start with $\Gamma[\sqrt{-3}]$.
The dimension of the space 
$S_{4,1+3k}(\Gamma[\sqrt{-3}],\text{det}^2)$ 
is given by (see \cite[Thm.\ 4.7]{B-vdG})
\[
\dim S_{4,1+3k}(\Gamma[\sqrt{-3}],\text{det}^2)=k(5k+1)/2-2
\]
for $k \geqslant 1$. 
One can show that $\dim S_{4,1,2}(\Gamma[\sqrt{-3}])=0$, for example by the following argument.
By multiplication with $E_6$ and the knowledge of $S_{4,7,2}(\Gamma[\sqrt{-3}])$
as  a $\mathfrak{S}_4$-module, we see that
only the $s[3,1]$ and $s[2,1,1]$ components can be non-zero.
Restricting to a component of $T_1$ is injective since such a component is the zero locus of
a form of weight $1$, and dividing would give a non-zero form of weight $(4,0)$ on some congruence subgroup.
The fact that $\dim S_3(\Gamma_0(3),\big(\frac{\cdot}{3}\big)))=2$ and $\dim s[3,1]=\dim s[2,1,1]=3$,
now shows that $\dim S_{4,1,2}(\Gamma[\sqrt{-3}])=0$.

The Hilbert--Poincar\'e series for the dimensions  
is therefore given by
\[
\sum_{k \geqslant 0} \dim S_{4,1+3k}(\Gamma[\sqrt{-3}],\text{det}^2)t^{1+3k}=
\frac{t^4+6\, t^7-2\, t^{10}}{(1-t^3)^3} \, .
\]
\begin{lemma}\label{4-7-10-lemma}
The $M(\Gamma[\sqrt{-3}])$-module $\Sigma_4^2(\Gamma[\sqrt{-3}])$
is generated by a generator of type $s[4]$ in weight $(4,4)$,
generators of type $s[3,1]$ and $s[2,1,1]$ in weight $(4,7)$
and a relation of type $s[2,2]$ in weight $(4,10)$.
\end{lemma}
\begin{proof}
This can be proved using results of \cite{B-vdG} as in \cite{C-vdG}.
\end{proof}
Writing the isotypical decomposition of $M_{3k}(\Gamma[\sqrt{-3}])=
{\rm Sym}^k(s[2,1,1])$ as 
$$
\Sym^k(s[2,1,1])=a_k\,s[4]+b_k\,s[3,1]+c_k\,s[2,2]+d_k\,s[2,1,1]+e_k\,s[1,1,1,1]
$$
we get by Lemma \ref{4-7-10-lemma} for $k\geqslant 1$
\[
S_{4,6k-2}(\Gamma[\sqrt{-3}],\text{det}^2)=(a_{2k-2}+b_{2k-3}+d_{2k-3}-c_{2k-4})\, s[4]+\ldots
\]
The generating series of the numbers 
$a_{k}$, $b_{k}$, $c_{k}$, $d_{k}$ and  $e_{k}$  
are given by the generating series 
$N/(1-t)(1-t^2)(1-t^3)(1-t^4)$ with $N$  as in the next table.

\begin{footnotesize}
\smallskip
\vbox{
\bigskip\centerline{\def\quad{\hskip 0.6em\relax}
\def\quod{\hskip 0.5em\relax }
\vbox{\offinterlineskip
\hrule
\halign{&\vrule#&\strut\quod\hfil#\quad\cr
height2pt&\omit&&\omit&&\omit&&\omit&&\omit&&\omit&\cr
&&& $a_k$ && $b_k$ && $c_k$ && $d_k$ && $e_k$ &\cr
\noalign{\hrule}
height2pt&\omit&&\omit&&\omit&&\omit&&\omit&&\omit&\cr
&$N$ && $(1-t)(1-t^3+t^6)$ && $t^2(1-t^3)$ && $(1-t)t^2(1+t^2)$ && $(t-t^2+t^3)(1-t^3)$ && $t^3(1-t)$ & \cr
} \hrule}
}}
\end{footnotesize}

\smallskip
\noindent
This leads to 
the following generating series for the dimension of the spaces $S_{4,6k-2}(\Gamma,\text{det}^2)$:
$$
\begin{aligned}
\sum_{k\geqslant 1} \dim S_{4,6k-2}(\Gamma,\text{det}^2)t^{6k-2}&= 
\frac{t^4+t^{10}+t^{16}+t^{22}+t^{28}}{(1-t^6)(1-t^{12})(1-t^{18})}\\
&=
t^4+2\, t^{10}+4\, t^{16}+7\, t^{22}+11\, t^{28}+O(t^{34})\, .
\end{aligned}
$$

Now we turn to the construction by covariants of the generators of weights
($4,4)$, $(4,10)$, $(4,16)$, $(4,22)$ and $(4,28)$.
 
The form of weight $(4,4)$ is already available:
$$
\chi_{4,4}= \frac{4900}{3\, c_1^4} \, \nu(J_{1,0,4}J_{1,4,0})
$$
and for later use we observe   that
the Fourier-Jacobi of its last component starts with

\begin{align*}
\chi_{4,4}^{(4)}(u,v)&=
X^2\big( 
q_v^{2}-6\,YZ\, q_v^{4} -20\,(Y^{3}+Z^{3})\, q_v^{5}+81\,Y^{2}Z^{2}\, q_v^{6}+132(Y^4Z+YZ^4)\, q_v^7\\
&+(122\, Y^6-800\, Y^3Z^3+122\, Z^6)\, q_v^8-(1020\, Y^7Z-1470\, Y^4Z^4+
1020\, YZ^7)\, q_v^{10}\\
&+(-76\, Y^9+1140\, Y^6Z^3+1140\, Y^3Z^6-76\, Z^9)\, q_v^{11}
+\ldots\big) \, .
 \end{align*}

Next we construct a generator $\chi_{4,10}$ of weight $10$. Note
that $\dim S_{4,10}(\Gamma,{\det}^2)=2$ and we know already a form
of weight $(4,10)$, namely
$E_6\,  \chi_{4,4}$. 

The three covariants 
$J_{2,4,0}J_{1,4,0}J_{1,0,4}$, $J_{1,4,0}^2J_{2,0,4}$ and
$J_{1,4,0}J_{0,1,1}J_{3,3,3}$ produce modular forms in $S_{4,10}(\Gamma,\text{det}^2)$.
But we have the following relation
$$
132\,J_{0,1,1}J_{3,3,3}+175(J_{2,4,0}J_{1,0,4}-J_{1,4,0}J_{2,0,4})=0\, .
$$
We know $\nu(J_{2,4,0})= - (c_1^6/5040 \gamma^2) E_6$. We set
$$
\chi_{4,10}=-\frac{2744000\, \gamma^2}{c_1^{10}}\nu(J_{1,4,0}^2J_{2,0,4}).
$$
One checks holomorphicity using the table in Section \ref{Construction}. 
The Fourier-Jacobi expansion of its last component starts with
\begin{align*}
\chi_{4,10}^{(4)}=&
X^2\big(
q_v^2+54\, YZ\, q_v^4-272(Y^3+Z^3)\, q_v^5+405\, Y^2Z^2\, q_v^6+
3024(Y^4Z+YZ^4)\, q_v^7\\
&+(4406\, Y^6-15560\, Y^3Z^3+4406\, Z^6)\, q_v^8-23328(Y^5Z^2+Y^2Z^5)\, q_v^9\\
&-(62748\, Y^7Z-221022\, Y^4Z^4+62748\, YZ^7)\, q_v^{10}\\
&-(22000\, Y^9-16368\, Y^6Z^3-16368\, Y^3Z^6+22000\, Z^9)\, q_v^{11}
+\ldots\big)\, .
 \end{align*}

The Fourier-Jacobi expansions of the last components of $E_6\chi_{4,4}$ and $\chi_{4,10}$ start with
$$
E_6\chi_{4,4}^{(4)}=X^2q_v^2+750\,X^2YZq_v^4+\ldots , \qquad
\chi_{4,10}^{(4)}=X^2q_v^2+54\,X^2YZq_v^4+\ldots \, .
$$
and this shows that they generate the space $S_{4,10}(\Gamma,\text{det}^2)$.

For the generator of weight $(4,16)$ we note that 
$\dim S_{4,16}(\Gamma,{\det}^2)=4$ and we have already 
three linearly independent elements 
$E_6^2 \chi_{4,4}$, $E_{12} \chi_{4,4}$ and $E_6\chi_{4,10}$.
We now put

$$
\chi_{4,16}=\frac{2304960000 \, \gamma^4}{c_1^{16}}
\nu(J_{1,4,0}^2(6\,J_{2,1,3}J_{2,3,1}-J_{2,0,0}J_{1,1,3}J_{1,3,1}))\, .
$$
We observe that the
Taylor expansion of $\nu(6\,J_{2,1,3}J_{2,3,1}-J_{2,0,0}J_{1,1,3}J_{1,3,1})$
along $T_1$ starts with
\begin{align*}
\nu&(6\,J_{2,1,3}J_{2,3,1}-J_{2,0,0}J_{1,1,3}J_{1,3,1})(u,\sqrt{-3}\tau)=\\
&\frac{c_1^8}{705600\gamma^6}
\Bigg(
\left[\begin{smallmatrix}
0 \\ 0 \\ \vartheta^4 \\ 0 \\ 0
 \end{smallmatrix}\right]u^{-2}+
 \gamma/6
\left[\begin{smallmatrix}
0 \\ 0 \\ 0 \\ \vartheta^3(3\vartheta^3+\vartheta e_2-108\psi^3)\\ 0
 \end{smallmatrix}\right]u^{-1} +\ldots
\Bigg) \, ,
\end{align*}
so the multiplication by $\nu(J_{1,4,0})^2$, proportional to $\zeta^2$, makes it holomorphic along $T_1$.

The Fourier-Jacobi expansion of the last component of $\chi_{4,16}$  starts with
\begin{align*}
\chi_{4,16}^{(4)}=&
X^2\big(
q_v^2+162\, YZq_v^4+3040(Y^3+Z^3)q_v^5+43497\, Y^2Z^2q_v^6-2592(Y^4Z+YZ^4)q_v^7\\
&-(298462\, Y^6+263600\, Y^3Z^3+298462\, Z^6)q_v^8-839808(Y^5Z^2+Y^2Z^5)q_v^9\\
&+(2185380\, Y^7Z-127170\, Y^4Z^4+2185380\, YZ^7)q_v^{10}\\
&+(4366688\, Y^9+8413152\, Y^6Z^3+8413152\, Y^3Z^6+4366688\, Z^9)q_v^{11}
+\dots \big)\, .
\end{align*}
The Fourier-Jacobi expansions of the last component of
$E_6^2\chi_{4,4}$, $E_{12}\chi_{4,4}$,  $E_6\chi_{4,10}$ and $\chi_{4,16}$ start with
\begin{align*}
E_6^2\chi_{4,4}^{(4)}&=X^2q_v^2 + 1506\,X^2YZq_v^4 + 4012\,X^2(Y^3+Z^3)q_v^5+603369\,X^2Y^2Z^2q_v^6+\ldots \\
E_{12}\chi_{4,4}^{(4)}&=X^2q_v^2 -3678\,X^2YZq_v^4 + 35116\,X^2(Y^3+Z^3)q_v^5+354537\,X^2Y^2Z^2q_v^6+\ldots\\
E_6\chi_{4,10}^{(4)}&=X^2q_v^2+810\,X^2YZq_v^4 + 1744\,X^2(Y^3+Z^3)q_v^5+61641\,X^2Y^2Z^2q_v^6+\ldots\\
\chi_{4,16}^{(4)}&=X^2q_v^2+162\,X^2YZq_v^4 + 3040\,X^2(Y^3+Z^3)q_v^5+43497\,X^2Y^2Z^2q_v^6+\ldots
\end{align*}
showing that these generate $S_{4,16}(\Gamma, {\det}^2)$.

Before we construct the last two generators we need a lemma.

\begin{lemma}\label{three_forms}
We have
$$
\begin{aligned}
\nu(J_{1,4,0}^2(J_{3,0,0}J_{1,3,1}-J_{2,0,0}J_{2,3,1}) & 
\in M_{1,16}(\Gamma)\, ,\\
\nu(J_{1,4,0}J_{1,1,3}) & \in M_{3,6}(\Gamma,{\det}^2)\, , \\
\nu(J_{1,4,0}J_{0,1,1}J_{3,4,2}) & \in M_{3,12}(\Gamma,{\det}^2) \, ,
\end{aligned}
$$
and these three forms are $\mathfrak{S}_4$-invariant.
\end{lemma}
For the proof one calculates the Taylor expansion along $T_1$ as done
for the examples above.

In order to get a form of weight $(4,22)$ we set
$$
\chi_{4,22}=\frac{53782400000\,\gamma^6}{3\, c_1^{22}}\nu(J_{1,4,0}^3J_{1,1,3}(J_{3,0,0}J_{1,3,1}-J_{2,0,0}J_{2,3,1}))
$$
and by applying Lemma \ref{three_forms}, 
we see that $\chi_{4,22}\in S_{4,22}(\Gamma,{\det}^2)$.

The Fourier-Jacobi expansion of its last component starts with
\begin{align*}
\chi_{4,22}^{(4)}=& X^2\big(
YZq_v^4+9(Y^3+Z^3)\, q_v^5+60\, Y^2Z^2q_v^6-277(Y^4Z+YZ^4)\, q_v^7\\
& -(6363\, Y^6-9468\, Y^3Z^3+6363\, Z^6)\, q_v^8
 +2106(Y^5Z^2+Y^2Z^5)\, q_v^9\\ 
&+(15128\, Y^7Z+27844\, Y^4Z^4+15128\, YZ^7)\, q_v^{10}+\\
&(276471\, Y^9-212895\, Y^6Z^3-212895\, Y^3Z^6+276471\, Z^9)\, q_v^{11}+\ldots
\big)\, .
\end{align*}
By using the Fourier-Jacobi of the last component of
$E_6^3\chi_{4,4}$, $E_6E_{12}\chi_{4,4}$, $E_9^2\chi_{4,4}$, $E_6^2\chi_{4,10}$, $E_{12}\chi_{4,10}$, $E_6\chi_{4,16}$
and $\chi_{4,22}$, we check that they are linearly independent 
so they span the space $S_{4,22}(\Gamma,{\det}^2)$ that is of
dimension $7$.

For the generator of weight $(4,28)$ we put
$$
\chi_{4,28}=
-\frac{51114792960000 \,\gamma^{8} }{c_{1}^{28}}
\nu(J_{1,4,0}^{3}J_{0,1,1} J_{3,4,2}\left(J_{3,0,0}J_{1,3,1} -J_{2,0,0} J_{2,3,1} \right))
$$
and by applying Lemma \ref{three_forms} 
we see that $\chi_{4,28}\in S_{4,28}(\Gamma,{\det}^2)$.
The Fourier-Jacobi expansion of its last component starts with
\begin{align*}
\chi_{4,28}^{(4)}=&
X^2\big(
YZ\, q_v^4+9(Y^3+Z^3)\, q_v^5-384\, Y^2Z^2q_v^6-7117(Y^4Z+YZ^4)\, q_v^7\\
&+(-31959\, Y^6-92592\, Y^3Z^3-31959\, Z^6)q_v^8-274698(Y^5Z^2+Y^2Z^5)\, q_v^9\\
&+(3511880\, Y^7Z-4338416\, Y^4Z^4+3511880\, YZ^7)\, q_v^{10}\\
&+(18226071\, Y^9-5450355\, Y^6Z^3-5450355\, Y^3Z^6+18226071\, Z^9)\, q_v^{11}
+\ldots
\big)\, .
\end{align*}
By using the Fourier-Jacobi of the last components of
$E_6^4\chi_{4,4}$, $E_6^2E_{12}\chi_{4,4}$, $E_6E_9^2\chi_{4,4}$, $E_{12}^2\chi_{4,4}$,
$E_{6}^3\chi_{4,10}$, $E_{6}E_{12}\chi_{4,10}$, $E_9^2\chi_{4,10}$,
$E_6^2\chi_{4,16}$, $E_{12}\chi_{4,16}$,
$E_6\chi_{4,22}$
and $\chi_{4,28}$, we check that they are linearly independent,
so they span the $11$-dimensional space $S_{4,28}(\Gamma,{\det}^2)$.

\begin{lemma}\label{wedge} The exterior product of our generators
satisfies
\[
\bigwedge_{k=1}^{5}
\chi_{4,6k-2}=
-\frac{c_1^{10}}{2\, \gamma^{10}}\, \rho^2\,
\zeta^{12}\,E_9^2 \, .
\]
\end{lemma}
\begin{proof}
We first note that the exterior product of the five forms  $\chi_{4,6k-2}$ for $k\in\{1,2,3,4,5\}$, which take values in
$\Sym^4(\CC^2)\simeq\CC^5$, can be viewed as the determinant 
of the five components of these five forms,
and viewing a covariant of degree 4 in $x_1, x_2$ as a vector of 
size $5$ whose $(i+1)$th component is the coefficient
of $x_1^{4-i}x_2^i$, $0\leqslant i  \leqslant 4$, we have
\begin{align*}
\quad \bigwedge_{k=1}^{5}
\chi_{4,6k-2}&=
2^{41}\cdot 3^2\cdot 5^{18}\cdot 7^{19}\cdot 11 \frac{\gamma^{20}}{c_1^{80}}\, 
\nu(J_{1,4,0})^{11} \, \times \\
\nu(\det(&
J_{1,0,4}J_{2,0,4}J_{1,1,3}J_{0,1,1} J_{3,4,2}
(6\,J_{2,1,3}J_{2,3,1}-J_{2,0,0}J_{1,1,3}J_{1,3,1})
(J_{3,0,0}J_{1,3,1}-J_{2,0,0}J_{2,3,1})^2))\\
&=-
2^{35}\cdot 3^2\cdot 5^{12}\cdot 7^{14}\cdot 11^2 \frac{\gamma^{20}}{c_1^{80}}\,
\nu(J_{1,4,0}^{12}J_{3,6,0}^2(J_{2,0,0}^3-6J_{3,0,0}^2)^2) \, .
\end{align*}
We have seen in Proposition \ref{nu-images} that
\[
\nu(J_{1,4,0)}=\frac{3 c_1^4}{70}\, \zeta, \quad
\nu(J_{3,6,0})=\frac{c_1^9}{798336\, \gamma^3} E_9, \quad 
\nu(32(J_{2,0,0}^3-6\,J_{3,0,0}^2))=-\rho \frac{c_{1}^{12}}{3^3\,\gamma^{12}}
\]
and this implies
\[
\bigwedge_{k=1}^{5}
\chi_{4,6k-2}=
-\frac{c_1^{10}}{2\, \gamma^{10}}\, \rho^2\,
\zeta^{12}\,E_9^2\, ,
\]
thus proving the lemma.

\end{proof}
We can now conclude the proof of Theorem \ref{Thm-Sigma-4-2}. 
The modular forms $\chi_{4,k}$
with $k\in \{ 4,10,16,22,28\}$ are algebraically independent
over $M(\Gamma)$ because of Lemma \ref{wedge}.
Since they generate a submodule with Hilbert--Poincar\'e 
series equal to that
of $\Sigma_4^2$ the result follows.
\end{proof}

\bigskip
In a similar way one can treat the cases $l=1$ and $l=0$. We intend to come back
to these cases in another paper.

\end{section}

\end{document}